\documentclass{daj}
\usepackage{amssymb}
\usepackage{amsthm}
\usepackage{amsmath}
\usepackage{mathtools}
\newcommand\E{\mathbb{E}}
\newcommand\Z{\mathbb{Z}}
\newcommand\Q{\mathbb{Q}}
\newcommand\R{\mathbb{R}}

\newcommand\F{\mathbb{F}}
\newcommand\C{\mathbb{C}}

\newcommand\N{\mathbb{N}}

\newcommand\X{\mathcal{X}}

\newcommand\eps{\varepsilon}
\newcommand\Bohr{\mathrm{Bohr}}
\newcommand\id{\operatorname{id}}
\newcommand\Aut{\operatorname{Aut}}
\newcommand\aderiv{{\partial}}
\newcommand\HK{\mathrm{HK}}
\newcommand\Poly{\mathrm{Poly}}
\newcommand\st{\mathrm{st}}
\newcommand\n{{\mathrm{n}}}
\newcommand\Hom{\operatorname{Hom}}

\newtheorem{theorem}{Theorem}[section]
\newtheorem{corollary}[theorem]{Corollary}
\newtheorem{conjecture}[theorem]{Conjecture}
\newtheorem{lemma}[theorem]{Lemma}
\newtheorem{proposition}[theorem]{Proposition}

\theoremstyle{definition}
\newtheorem{definition}[theorem]{Definition}
\newtheorem{remark}[theorem]{Remark}
\newtheorem{example}[theorem]{Example}

\makeatletter
\newcommand{\subalign}[1]{%
  \vcenter{%
    \Let@ \restore@math@cr \default@tag
    \baselineskip\fontdimen10 \scriptfont\tw@
    \advance\baselineskip\fontdimen12 \scriptfont\tw@
    \lineskip\thr@@\fontdimen8 \scriptfont\thr@@
    \lineskiplimit\lineskip
    \ialign{\hfil$\m@th\scriptstyle##$&$\m@th\scriptstyle{}##$\hfil\crcr
      #1\crcr
    }%
  }%
}
\makeatother

\dajAUTHORdetails{%
  title = {The Inverse Theorem for the $U^3$ Gowers Uniformity Norm on Arbitrary Finite Abelian Groups: Fourier-analytic and Ergodic Approaches}, 
  author = {Asgar Jamneshan, and Terence Tao},
  plaintextauthor = {Asgar Jamneshan, and Terence Tao},
  plaintexttitle = {The Inverse Theorem for the third Gowers Uniformity Norm on Arbitrary Finite Abelian Groups: Fourier-analytic and Ergodic Approaches}, 
  runningtitle = {Inverse theorem for $U^3(G)$}, 
}   

\dajEDITORdetails{%
   year={2023},
   number={11},
   received={17 March 2022},   
   published={31 July 2023},  
   doi={10.19086/da.84268},       
}   

\begin{document}

\begin{frontmatter}[classification=text]


\author[aj]{Asgar Jamneshan\thanks{Supported by DFG-research fellowship JA 2512/3-1.}}
\author[tao]{Terence Tao\thanks{Supported by a Simons Investigator grant, the James and Carol Collins Chair, the Mathematical Analysis \& Application Research Fund Endowment, and by NSF grant DMS-1764034.}}

\begin{abstract}
We state and prove a quantitative inverse theorem for the Gowers uniformity norm $U^3(G)$ on an arbitrary finite abelian group $G$; the cases when $G$ was of odd order or a vector space over ${\mathbf F}_2$ had previously been established by Green and the second author and by Samorodnitsky respectively by Fourier-analytic methods, which we also employ here.  We also prove a qualitative version of this inverse theorem using a  structure theorem of Host--Kra type for ergodic $\Z^\omega$-actions of order $2$ on probability spaces established recently by Shalom and the authors. 
\end{abstract}
\end{frontmatter}


\section{Introduction}

In this paper we investigate the inverse theory for the third Gowers norm $U^3(G)$ on arbitrary finite abelian groups $G = (G,+)$, with a particular focus on results that apply uniformly for \emph{all} such groups $G$ without additional hypotheses.

\subsection{Gowers uniformity norms}

We first recall the definition of the Gowers uniformity norms, first introduced in \cite{gowers2} to obtain a new proof of Szemer\'edi's theorem \cite{szemeredi1975sets} on arithmetic progressions.  Throughout this paper we use the usual averaging notation
$$ \E_{x \in A} f(x) \coloneqq \frac{1}{|A|} \sum_{x \in A} f(x)$$
for a function $f \colon A \to \C$ defined on a finite non-empty set $A$ of some cardinality $|A|$.

\begin{definition}[Uniformity norms]  Let $G = (G,+)$ be a finite additive\footnote{In this paper all additive groups are assumed to be abelian; multiplicative groups $G = (G,\cdot)$ are not required to be abelian.} group.
\begin{itemize}
\item[(i)]  If $d \geq 0$ and $(f_\omega)_{\omega \in \{0,1\}^d}$ is a tuple of functions $f_\omega \colon G \to \C$, we define the \emph{Gowers inner product} $\langle (f_\omega)_{\omega \in \{0,1\}^d} \rangle_{U^d(G)}$ by the formula
\begin{equation}\label{unif-def}
\langle (f_\omega)_{\omega \in \{0,1\}^d} \rangle_{U^d(G)} \coloneqq \E_{x,h_1,\dots,h_d \in G} \prod_{\omega \in \{0,1\}^d} {\mathcal C}^{|\omega|} f_\omega(x + \omega \cdot (h_1,\dots,h_d))
\end{equation}
where $\omega = (\omega_1,\dots,\omega_d)$, $|\omega| \coloneqq \omega_1+\dots+\omega_d$, ${\mathcal C} \colon z \mapsto \overline{z}$ is the complex conjugation map, and we use the dot product notation
$$ (n_1,\dots,n_d) \cdot (v_1,\dots,v_d) \coloneqq n_1 v_1 + \dots + n_d v_d$$
whenever $n_1,\dots,n_d$ are integers and $v_1,\dots,v_d$ lie in an additive group.
\item[(ii)] If $d \geq 1$ and $f \colon G \to \C$ is a function, the \emph{Gowers uniformity norm} $\|f\|_{U^d(G)}$ is defined by the formula
$$ \|f\|_{U^d(G)} \coloneqq \langle (f)_{\omega \in \{0,1\}^d} \rangle_{U^d(G)}^{1/2^d}.$$
\end{itemize}
\end{definition}

\begin{example} One has
\begin{equation}\label{u1-form}
\begin{split}
\|f\|_{U^1(G)} &= \left| \E_{x,h \in G} f(x) \overline{f}(x+h) \right|^{1/2} \\
&= |\E_{x \in G} f(x)|
\end{split}
\end{equation}
and
$$
\|f\|_{U^2(G)} = \left| \E_{x,h_1,h_2 \in G} f(x) \overline{f}(x+h_1) \overline{f}(x+h_2) f(x+h_1+h_2) \right|^{1/4}. $$
We also have the recursive formula
\begin{equation}\label{gowers-recurse}
\|f\|_{U^{d+1}(G)} = \left( \E_{h \in G} \| \Delta^\times_h f \|_{U^d(G)}^{2^d} \right)^{1/2^{d+1}}
\end{equation}
for any $d \geq 1$, where $\Delta^\times_h f(x) \coloneqq f(x) \overline{f(x-h)}$ denotes the ``multiplicative derivative'' of $f$ in the direction $h$.  In a similar spirit, we observe the recursive formula
\begin{equation}\label{inner-recurse}
\langle (f_\omega)_{\omega \in \{0,1\}^{d+1}} \rangle_{U^{d+1}(G)} = \E_{h^0,h^1 \in G}
\langle (f_{\omega,0}(\cdot-h^0) \overline{f_{\omega,1}}(\cdot-h^1))_{\omega \in \{0,1\}^d} \rangle_{U^d(G)}
\end{equation}
which follows after expanding out both sides and applying a change of variables.
\end{example}

One can verify that $\|\cdot \|_{U^d(G)}$ is a semi-norm for $d=1$ and a norm for $d > 1$; see for instance \cite[\S 11.1]{tao-vu} for the basic properties of these norms.  Our focus in this paper will primarily be on the $U^3(G)$ norm.

If a function $f \colon G \to \C$ is \emph{$1$-bounded} in the sense that $|f(x)| \leq 1$ for all $x \in G$, then clearly we have
$$ \|f\|_{U^{s+1}(G)} \leq 1$$
for any $s \geq 1$.  For many applications to additive combinatorics, it is of interest to establish an \emph{inverse theorem} that describes when a converse inequality
$$ \|f\|_{U^{s+1}(G)} \geq \eta$$
holds for some $0 < \eta \leq 1$.  Here one is primarily interested in the regime where $\eta$ is fixed, but the order $|G|$ of the group can be extremely large. The inverse theorem for $U^{s+1}(G)$ is a foundational result in ``degree $s$ Fourier analysis''; see for instance \cite{green-quadratic}, \cite[\S 11]{tao-vu} for further discussion.  

In the case $s=1$ of linear Fourier analysis, the inverse theorem is easy to state and prove.  Recall that the \emph{Pontragyin dual group} $\hat G$ of a finite additive group is the collection of all additive homomorphisms $\xi \colon G \to \R/\Z$ from $G$ to $\R/\Z$; we write $\xi \cdot x$ for $\xi(x)$, and refer to elements $\xi$ of $\hat G$ as \emph{frequencies}.

\begin{theorem}[Inverse theorem for $U^2(G)$]\label{inverse-u2}   Let $G$ be a finite additive group, let $0 < \eta \leq 1$, and let $f \colon G \to \C$ be a $1$-bounded function such that $\|f\|_{U^2(G)} \geq \eta$.  Then there exists $\xi \in \hat G$ such that
$$ |\E_{x \in G} f(x) e(-\xi \cdot x)| \geq \eta^2$$
where $e \colon \R/\Z \to \C$ is the standard character $e(\theta) \coloneqq e^{2\pi i \theta}$.
\end{theorem}

\begin{proof} See e.g., \cite[Proposition 2.2]{green-tao-inverseu3}.
\end{proof}

\subsection{$U^3(G)$ inverse theorem via quadratic Fourier analysis}

The inverse theorems in the higher order case $s>1$ are harder to describe.  In the $s=2$ case, one way to formulate an inverse theorem (following \cite[Theorem 2.7]{green-tao-inverseu3}) is to replace the linear phases $x \mapsto e(\xi \cdot x)$ in Theorem \ref{inverse-u2} by ``locally quadratic phase functions''.  To describe this formulation more precisely, we need to recall some standard definitions.

\begin{definition}[Bohr sets] (See e.g., \cite[\S 4.4]{tao-vu}) Let $G$ be a finite additive group.
\begin{itemize}
\item[(i)]  If $S \subset \hat G$ is a set of frequencies, we define a ``seminorm'' $\| \|_S \colon G \to [0,1/2]$ on $G$ by
$$ \| x\|_S \coloneqq \sup_{\xi \in S} \| \xi \cdot x\|_{\R/\Z}$$
where $\|\theta\|_{\R/\Z}$ denotes the distance to the nearest integer.  The \emph{Bohr set} $\Bohr(S,\rho) \subset G$ is then defined as
$$ \Bohr(S,\rho) \coloneqq \{ x \in G: \|x\|_S < \rho \}.$$
\item[(ii)]  A Bohr set $\Bohr(S,\rho)$ is \emph{regular} if one has
$$ (1 - 100 |S| |\kappa|) |\Bohr(S,\rho)| \leq |\Bohr(S,(1+\kappa) \rho)| \leq (1 + 100 |S| |\kappa|) |\Bohr(S,\rho)|$$
whenever $\kappa \in [-\frac{1}{100|S|}, \frac{1}{100|S|}]$.
\end{itemize}
\end{definition}

\begin{definition}[Locally quadratic functions]  Let $U$ be a subset of an additive group $G$, and let $H$ be another additive group.
\begin{itemize}
\item[(i)]  If $\phi \colon U \to H$ is a map from $U$, and $h \in G$, the derivative $\aderiv_h \phi \colon U \cap (U+h) \to H$ is defined by $\phi(x) \coloneqq \phi(x) - \phi(x-h)$.  (For some values of $h$, $\aderiv_h \phi$ may be the empty function.)
\item[(ii)]  A map $\phi \colon U \to H$ is said to be \emph{locally linear} (resp. \emph{locally quadratic}) on $U$ if the second derivatives $\aderiv_h \aderiv_k \phi$ (resp. the third derivatives $\aderiv_h \aderiv_k \aderiv_l \phi$) vanish identically for all $h,k,l \in G$.  (Here we adopt the convention that the empty function vanishes identically.)
\end{itemize}
\end{definition}

The first main result of our paper is the following inverse theorem for the $U^3(G)$ norm.  Our conventions for asymptotic notation are laid out in Section \ref{notation-sec} below.

\begin{theorem}[Inverse theorem for $U^3(G)$, locally quadratic version]\label{inverse-u3} Let $G$ be a finite additive group, let $0 < \eta \leq 1/2$, and let $f \colon G \to \C$ be a $1$-bounded function with $\|f\|_{U^3(G)} \geq \eta$.  Then there exists a regular Bohr set $B(S,\rho)$ with $|S| \ll \eta^{-O(1)}$ and $\exp(-\eta^{-O(1)}) \ll \rho \leq 1/100$, a locally quadratic function $\phi \colon \Bohr(S,\rho) \to \R/\Z$, and a function $\xi \colon G \to \hat G$ such that
\begin{equation}\label{exh}
\E_{x \in G} |\E_{h \in \Bohr(S,\rho)} f(x+h) e(-\phi(h)-\xi(x) \cdot h)| \gg \eta^{O(1)}.
\end{equation}
\end{theorem}

In the case where $|G|$ is odd, this theorem was essentially established in \cite[Theorem 2.7]{green-tao-inverseu3}; when instead $G$ is of the form $G = \F_2^n$ for some $n$, this result was essentially established in \cite[Theorem 2.3]{samorodnitsky}.  In both cases the result was established by the methods of ``quadratic Fourier analysis'' \cite{green-quadratic}, which in particular utilised tools from additive combinatorics such as the Balog--Szemer\'edi--Gowers theorem.  In these arguments, the restrictions on $G$ arise at a key step when one wishes to express a certain symmetric (locally) bilinear form in terms of a (locally) quadratic phase function, and this is achieved either by taking advantage of the ability to divide by two (in the case when $|G|$ is odd) or by working with an explicit basis for $G$ (in the case when $G = \F_2^n$).  In Section \ref{Fourier-sec} we establish Theorem \ref{inverse-u3} in full generality.  The proof uses the same quadratic Fourier analysis methods employed in \cite{green-tao-inverseu3}, \cite{samorodnitsky}; to achieve the key step for arbitrary $G$, we introduce a lifting lemma that converts a locally bilinear form on $G$ to a globally bilinear form on a larger group $G_S$, and then works using an explicit basis in the latter group to describe the latter bilinear form in terms of a (globally) quadratic form, which then descends to give the locally quadratic function $\phi$ in the above theorem.

\subsection{Nilmanifold formulation}

It is now generally accepted that the most natural formulation of $U^{s+1}(G)$ inverse theorems for higher values of $s$ involves nilmanifolds and related objects. In order to state such a formulation, we recall some key definitions.

\begin{definition}[Host-Kra spaces]  (See e.g., \cite[Appendix B]{gtz}.)
\begin{itemize}
\item[(i)]  A \emph{prefiltration} on a group $G = (G,\cdot)$ is a collection of subgroups $G_\bullet = (G_i)_{i=0}^\infty$ of $G$ such that $G_0 \geq G_1 \geq \dots$ and $[G_i,G_j] \subset G_{i+j}$ for all $i,j \geq 0$.  We say that the prefiltration has \emph{degree at most $s$} if $G_i$ is trivial for $i>s$.  We refer to the pair $(G,G_\bullet)$ as a \emph{prefiltered group}.  If $G_0=G_1=G$, we say that the prefiltration is a \emph{filtration}, and that $(G,G_\bullet)$ is a \emph{filtered group}.  Similarly if the group $G$ is written additively instead of multiplicatively.  Note that filtered groups of degree at most $s$ are automatically nilpotent of nilpotency step at most $s$.  If all the groups $G_i$ in a filtered group are locally compact, we call $G$ a \emph{filtered locally compact group}.  A subgroup $\Gamma$ of a filtered locally compact group $G$ is a \emph{filtered lattice} if $\Gamma$ is discrete and $\Gamma_i \coloneqq \Gamma \cap G_i$ is a cocompact subgroup of $G_i$ for all $i$; thus $(\Gamma, (\Gamma_i)_{i=0}^\infty)$ is a filtered discrete group.
\item[(ii)]  If $G = (G,+)$ is an additive group, we define the \emph{abelian filtration} on $G$ by setting $G_0 = G_1 \coloneqq G$ and $G_i = \{0\}$ for $i>1$.
\item[(iii)]  If $G = (G,G_\bullet)$ is a prefiltered group, and $k \geq 0$, we define the \emph{Host-Kra group} $\HK^k(G) = \HK^k(G,G_\bullet)$ to be the subgroup of $G^{\{0,1\}^k}$ generated by the elements
\begin{equation}\label{go}
[g_{\omega_0}]_{\omega_0} \coloneqq \left(g_{\omega_0}^{1_{\omega \geq \omega_0}}\right)_{\omega \in \{0,1\}^k}
\end{equation}
for $\omega_0$ in $\{0,1\}^k$ and $g_{\omega_0} \in G_{|\omega_0|}$, where $g_{\omega_0}^{1_{\omega \geq \omega_0}}$ is defined to equal $g_{\omega_0}$ when $\omega \geq \omega_0$ (in the product order on $\{0,1\}^k$) and the identity $1$ otherwise.
\item[(iv)]  If $G = (G,G_\bullet)$ is a prefiltered group, $k \geq 0$, and $\Gamma$ is a subgroup of $G$, we define the \emph{Host-Kra space} $\HK^k(G/\Gamma)$ to be the set
$$ \HK^k(G/\Gamma) \coloneqq \pi^{\{0,1\}^k}( \HK^k(G) ) \subset (G/\Gamma)^{\{0,1\}^k},$$
where $\pi \colon G \to G/\Gamma$ is the quotient map and $\pi^{\{0,1\}^k} \colon G^{\{0,1\}^k} \to (G/\Gamma)^{\{0,1\}^k}$ is the map defined by pointwise evaluation of $\pi$, thus
$$ \pi^{\{0,1\}^k}( (h_\omega)_{\omega \in \{0,1\}^k} ) \coloneqq (\pi(h_\omega))_{\omega \in \{0,1\}^k}.$$
Note that we have a canonical identification
$$ \HK^k(G/\Gamma) \equiv \HK^k(G) / \HK^k(\Gamma)$$
and also
$$ \HK^k(\Gamma) = \HK^k(G) \cap \Gamma^{\{0,1\}^k}.$$
\item[(v)]  If $G = (G,G_\bullet)$ and $H = (H,H_\bullet)$ are prefiltered groups, and $\Gamma$, $\Lambda$ are subgroups of $G,H$ respectively, a \emph{nilspace morphism}\footnote{For a systematic treatment of nilspaces, of which the filtered groups $G$ and quotients $H/\Lambda$ are special cases, see \cite{candela0}.  However, we will not require the general theory of nilspaces here, as we will work only with filtered groups and their quotients.} $\phi \colon G/\Gamma \to H/\Lambda$ is a function such that for every $k \geq 0$, the map $\phi^{\{0,1\}^k} \colon (G/\Gamma)^{\{0,1\}^k} \to (H/\Lambda)^{\{0,1\}^k}$ defined by the pointwise application of $\phi$, thus
$$ \phi^{\{0,1\}^k}( (x_\omega)_{\omega \in \{0,1\}^k} ) \coloneqq (\phi(x_\omega))_{\omega \in \{0,1\}^k}$$
maps $\HK^k(G/\Gamma)$ to $\HK^k(H/\Lambda)$.  A nilspace morphism from a prefiltered group $G$ to a quotient $H/\Lambda$ will be called a \emph{polynomial map}.  (For instance, the projection $\pi \colon H \to H/\Lambda$ from (iv) is a polynomial map.)
\end{itemize}
\end{definition}

For future reference we note that if $G = (G,+)$ has the abelian filtration and $H = (H,\cdot)$ is a filtered group, then a map $\phi \colon G \to H$ is a polynomial map if and only if one has the relations
\begin{equation}\label{relate}
\aderiv_{h_1} \dots \aderiv_{h_k} \phi(x) \in H_k
\end{equation}
for all $k \geq 0$ and $x,h_1,\dots,h_k \in G$, where the difference operators $\aderiv_h$ are defined as
$$ \aderiv_h \phi(x) \coloneqq \phi(x) \phi(x-h)^{-1};$$
see for instance \cite[Theorem B.3]{gtz}.

\begin{definition}[Nilmanifolds]\cite{green-tao-nilmanifolds}  Let $s \geq 1$. A \emph{filtered nilmanifold} of degree at most $s$ is a quotient $H/\Gamma$, where $H = (H,H_\bullet)$ is a filtered Lie group of degree at most $s$, and $\Gamma=(\Gamma,\Gamma_\bullet)$ is a filtered lattice in $H$.  A \emph{prefiltered nilmanifold} is defined similarly, but with all filtrations replaced by prefiltrations.
\end{definition}

It is not difficult to see that a filtered nilmanifold $H/\Gamma$ has the structure of a smooth compact manifold (and the Host--Kra spaces $\HK^k(H/\Gamma)$ are also smooth compact manifolds); see for instance \cite[\S 2]{gmv}.  On the other hand, we stress that we do \emph{not} require a nilmanifold to be connected.  Nilsequences $x \mapsto F(g(x))$ arising from prefiltered nilmanifolds, with $g \colon G \to H/\Gamma$ a polynomial map and $F \colon H/\Gamma \to \C$ a Lipschitz map can be converted to nilsequences on a filtered nilmanifold by observing that $g(x) \in H_1 g(0)$ for all $x \in G$, so if one lets $\{g(0)\} \in H$ be a bounded element of $H$ that maps $g(0)$ to the origin, then $g(x) = \{g(0)\} \tilde g(x)$ for some polynomial map $\tilde g \colon G \to H_1 / (\Gamma \cap H_1)$ into the filtered nilmanifold $H_1/(\Gamma \cap H_1)$.  One can then write $F(g(x)) = \tilde F(\tilde g(x))$ where $\tilde F \colon H_1/(\Gamma \cap H_1) \to \C$ is the Lipschitz map $\tilde F(z) \coloneqq F(\{g(0)\} z)$.

For the $U^3$ inverse theory, we will work with the following specific construction of a degree $2$ (pre-)filtered nilmanifold, which one can view as a variant of the Heisenberg nilmanifold:

\begin{definition}[Specific nilmanifold construction]\label{spef} Let $N \geq 0$ be a natural number.  For any $d \geq 0$, we let $\Poly_{\leq d}(\R^N \to \R)$ denote the vector space of real valued polynomials on $\R^N$ of degree at most $d$.  These spaces have a translation action of $\R^N$, and in particular we can form the semidirect product
$$ H(\R^N) \coloneqq \R^N \ltimes \Poly_{\leq 2}(\R^N \to \R)$$
which is the space of pairs $(x, \phi)$ with $x \in \R^N$ and $\phi \colon \R^N \to \R$ a quadratic polynomial, with group law
$$ (x,\phi) (y,\psi) \coloneqq (x+y, T^y \phi + \psi)$$
where $T^x$ denotes the translation operation $T^x \psi(y) \coloneqq \psi(y+x)$.  This is a nilpotent Lie group (of nilpotency class three), and can be given the degree $2$ prefiltration
\begin{align*}
 H(\R^N)_0 &\coloneqq \R^N \ltimes \Poly_{\leq 2}(\R^N \to \R) \\
 H(\R^N)_1 &\coloneqq \R^N \ltimes \Poly_{\leq 1}(\R^N \to \R) \\
 H(\R^N)_2 &\coloneqq 0 \ltimes \Poly_{\leq 0}(\R^N \to \R) \\
 H(\R^N)_i &\coloneqq 0 \hbox{ for all } i > 2.
 \end{align*}
Inside $H(\R^N)$ we have a filtered lattice
$$ H(\Z^N) \coloneqq \Z^N \ltimes \Poly_{\leq 2}(\Z^N \to \Z)$$
defined and filtered in the analogous fashion (replacing all occurrences of $\R$ with $\Z$), using interpolation to identify $ \Poly_{\leq 2}(\Z^N \to \Z)$ with a subgroup of $\Poly_{\leq 2}(\R^N \to \R)$.  Thus $H(\R^N)/H(\Z^N)$ is a prefiltered nilmanifold (of dimension $\frac{N^2+5N+2}{2}$).  We can give it a (somewhat artificial) metric by identifying the Lie algebra of $H(\R^N)$ (as a vector space) with $\R^N \times \Poly_{\leq 2}(\R^N \to \R)$ and using the standard basis of $\R^N$ and the monomial basis of $\Poly_{\leq 2}(\R^N \to \R)$ as an orthonormal basis for this Lie algebra, which then defines a right-invariant Riemannian metric on $H(\R^N)$, and hence a Riemannian metric on $H(\R^N)/H(\Z^N)$.
\end{definition}

In Section \ref{tonil-sec} we will use Theorem \ref{inverse-u3} (or more precisely a slight strengthening of this result that involves a globally quadratic function on a lift of $G$) to establish the following inverse theorem, whose formulation is inspired by the recent work of Candela and Szegedy \cite{candela-szegedy-inverse} (which we will discuss shortly):

\begin{theorem}[Inverse theorem for $U^3(G)$, nilmanifold version]\label{inverse-u3-nil} Let $G$ be a finite additive group, let $0 < \eta \leq 1/2$, and let $f \colon G \to \C$ be a $1$-bounded function with $\|f\|_{U^3(G)} \geq \eta$.  
Then there exists a natural number $N \ll \eta^{-O(1)}$, a polynomial map $g \colon G \to H(\R^N)/H(\Z^N)$, and a Lipschitz function $F: H(\R^N) / H(\Z^N) \to \C$ of norm $O(\exp(\eta^{-O(1)}))$ such that
\begin{equation}\label{exf}
 |\E_{x \in G} f(x) \overline{F(g(x))}| \gg \exp(-\eta^{-O(1)}).
 \end{equation}
\end{theorem}

By the preceding discussion one can represent the prefiltered nilsequence $F(g(x))$ as a filtered nilsequence.  This theorem is somewhat weaker than Theorem \ref{inverse-u3} in that the correlation in \eqref{exf} is only exponential in $\eta$ instead of polynomial.  With a bit more effort one could create a more complicated ``averaged correlation'' variant of \eqref{exf}  (resembling \eqref{exh}), in which the lower bound in the correlation is polynomial in $\eta$, but we will not do so here.

\subsection{$U^3(G)$ inverse theorem via ergodic theory}

Based on Theorem \ref{inverse-u3}, we now formulate the following conjecture\footnote{One could similarly formulate a conjectural qualitative analogue of Theorem \ref{inverse-u3} to higher values of $s$, but we will not do so here, in part because one would have to decide precisely how to define the notion of a regular higher order Bohr set.} for general $G$ and $s$:

\begin{conjecture}[Inverse theorem for $U^{s+1}(G)$]\label{uk-inverse-nil}  Let $G$ be a finite additive group, let $\eta > 0$, let $s \geq 1$, and let $f \colon G \to \C$ be a $1$-bounded function with $\|f\|_{U^{s+1}(G)} \geq \eta$.  Then there exists a degree $s$ filtered nilmanifold $H/\Gamma$, drawn from some finite collection ${\mathcal N}_{s,\eta}$ of such nilmanifolds that depends only on $s,\eta$ but not on $G$ (and each such nilmanifold in ${\mathcal N}_{s,\eta}$ is endowed arbitrarily with a smooth Riemannian metric), a Lipschitz function $F \colon H/\Gamma \to \C$ of Lipschitz norm $O_{\eta,s}(1)$, and a polynomial map $g \colon G \to H/\Gamma$ such that
\begin{equation}\label{fF}
 |\E_{x \in G} f(x) \overline{F(g(x))}| \gg_{\eta,s} 1.
 \end{equation}
\end{conjecture}

There is substantial evidence already towards this conjecture:

\begin{itemize}
\item[(i)]  Theorem \ref{inverse-u3-nil} clearly establishes the $s=2$ case of this conjecture, and Theorem \ref{inverse-u2} similarly establishes the $s=1$ case.
\item[(ii)]  Conjecture \ref{uk-inverse-nil} is known\footnote{In the earlier references an ostensibly weaker version of the conjecture is established in which one works with ``non-periodic nilsequences'' rather than ``periodic nilsequences'', but the two formulations can be shown to be equivalent: see \cite{manners-periodic} for further discussion.} when $G$ is a cyclic group \cite{gtz-4}, \cite{gtz}, \cite{szegedy-inv2}, \cite{szegedy-inv}, \cite{manners-inv}, \cite{candela-szegedy-inverse}.  We remark that this special case has applications to number theory, for instance in obtaining asymptotics for linear equations in primes \cite{gt-linear}.
\item[(iii)] Conjecture \ref{uk-inverse-nil} is known when $G$ is of the form $\F_p^n$ for some finite field $\F_p$ of fixed prime order \cite{tz-finite}, \cite{btz}, \cite{tz-lowchar}, \cite{szegedy-inv2}, \cite{milicevic}, \cite{milicevic-2}, \cite{BSST}, \cite{CGSS}.  In this case one can take the nilmanifold $H/\Gamma$ to be the unit circle $\R/\Z$, or even a cyclic group $\frac{1}{p^j} \Z/\Z$ for some $j$ depending only on $p,s$, equipped with a suitable degree $s$ filtration.
\item[(iv)]  The inverse theorem in \cite[Theorem 1.6]{candela-szegedy-inverse} implies (as a special case) a weaker version of Conjecture \ref{uk-inverse-nil} for all $G,s$, in which the nilmanifold $H/\Gamma$ is replaced by the more general notion of a ``CFR nilspace''.  The main difficulty in passing from this theorem to the full strength version of Conjecture \ref{uk-inverse-nil} is that the CFR nilspaces produced by that theorem need not be connected (or ``toral'') without further hypotheses\footnote{For instance, as remarked previously, when $G = \F_p^n$, the CFR nilspace can be chosen to be isomorphic to the disconnected nilmanifold $\frac{1}{p^j} \Z/\Z$; see for instance \cite{tz-lowchar}.}  on the group $G$.
\item[(v)] In the converse direction, one can adapt standard arguments (e.g., \cite[Proposition 12.6]{green-tao-inverseu3}) to show that if $f \colon G \to \C$ is $1$-bounded and $|\E_{x \in G} f(x) \overline{F(g(x))}| \geq \delta$ for some $\delta>0$, degree $s$ filtered nilmanifold $H/\Gamma$, some Lipschitz function $F \colon H/\Gamma \to \C$ of norm at most $1/\delta$, and some polynomial map $g \colon G \to H/\Gamma$, then $\|f\|_{U^{s+1}(G)} \gg_{\delta,H/\Gamma,s} 1$.   
\end{itemize}

\begin{remark} From item (v) we see that Conjecture \ref{uk-inverse-nil} is an equivalence at the qualitative level.  However, it may be still possible to ``strengthen'' this conjecture by placing additional properties on $H/\Gamma$ and $g$ in certain cases.  For instance, the results in \cite{green-tao-nilmanifolds}, \cite{candela-szegedy-inverse} strongly suggest that the map $g$ can be chosen to obey a suitable ``balance'' or ``equidistribution'' property.  If the group $G$ is the product of boundedly many cyclic factors, one should be able to place additional connectedness or ``toral'' hypotheses on the nilmanifold $H/\Gamma$, in analogy with the analysis in \cite{host2005nonconventional} or \cite{candela-szegedy-inverse}.  In the case $G = \F_p^n$, it was recently established in \cite{CGSS} that the nilmanifold $H/\Gamma$ can be chosen to obey a ``$p$-homogeneity'' property.  We will not pursue the question of what the optimal properties one can place on $H/\Gamma$ or $g$ further in this paper.
\end{remark}

We will not establish this conjecture in full in this paper; however in the $s=2$ case we can establish it via an ergodic theory result (Theorem \ref{hk-inverse} below) that was recently established in \cite{jt21} by Shalom and the authors. In order to state this result, we will need some definitions.

\begin{definition}[$\Z^\omega$-systems]\ 
\begin{itemize}
\item[(i)]  We define the additive group $\Z^\omega$ to be the free abelian group by a countable sequence of generators $e_1,e_2,\dots$. This is an amenable group; it will be convenient to use the specific F{\o}lner sequence $\Phi_1, \Phi_2, \dots \subset \Z^\omega$ defined by
\begin{equation}\label{folner}
\Phi_n \coloneqq \{ a_1 e_1 + \dots + a_{2^n} e_{2^n}: a_1,\dots,a_{2^n} \in \{0,\dots,n\} \}.
\end{equation}
\item[(ii)]  A \emph{$\Z^\omega$-system} $(X,\mu,T)$ is a measure space $(X,\mu) = (X,\X,\mu)$, together with a group homomorphism $T \colon h \mapsto T^h$ from $\Z^\omega$ to the automorphism group $\Aut(X,\mu)$ of $(X,\mu)$ (the group of invertible measure-preserving transformations of $(X,\mu)$).  Note that any sequence $T^{e_1},T^{e_2},\dots$ of commuting automorphisms of $(X,\mu)$ will generate such a system.  For any $f \in L^\infty(X,\mu)$ and $h \in \Z^\omega$, we write $T^h f \coloneqq f \circ T^{-h}$.  A $\Z^\omega$-system is \emph{ergodic} if the invariant factor $L^\infty(X,\mu)^T \coloneqq \bigcap_{h \in \Z^\omega} \{ f \in L^\infty(X,\mu): T^h f = f \}$ consists only of the constants. If the $\sigma$-algebra $\X$ is countably generated modulo null sets, we say that the $\Z^\omega$-system is \emph{separable}.
\item[(iii)]  Let $(X,\mu)$ be a $\Z^\omega$-system, $d \geq 0$, and $(f_\omega)_{\omega \in \{0,1\}^d}$ be a tuple of functions $f_\omega \in L^\infty(X,\mu)$.  For any natural numbers $n_1,\dots,n_d$, we define the local Host--Kra inner product
\begin{equation}\label{bo}
\langle (f_\omega)_{\omega \in \{0,1\}^d} \rangle_{U^d_{n_1,\dots,n_d}(X)}
\coloneqq \E_{h^0_i, h^1_i \in \Phi_{n_i} \forall i=1,\dots,d} \int_X \prod_{\omega \in \{0,1\}^d} {\mathcal C}^{|\omega|} T^{h^{\omega_1}_1 + \dots + h^{\omega_d}_d} f_\omega\ d\mu
\end{equation}
and the global Host--Kra inner product
\begin{equation}\label{bo-global}
 \langle (f_\omega)_{\omega \in \{0,1\}^d} \rangle_{U^d(X)} \coloneqq \lim_{(n_1,\dots,n_d) \to \infty}
\langle (f_\omega)_{\omega \in \{0,1\}^d} \rangle_{U^d_{n_1,\dots,n_d}(X)}
\end{equation}
(for a proof of convergence of this limit and its independence of the choice of F{\o}lner sequence $\Phi_n$, see \cite[Theorem 2.1, Remark 1.14]{tz-concat}.)  For any $f \in L^\infty(X,\mu)$ and $d \geq 1$, we define the \emph{Host--Kra seminorm}
$$ \|f\|_{U^d(X)} \coloneqq \langle (f)_{\omega \in \{0,1\}^d} \rangle_{U^d(X)}^{1/2^d}.$$ 
\item[(iv)]  A \emph{$\Z^\omega$-morphism} $\pi \colon X \to Y$ from one $\Z^\omega$-system $(X,\mu_X,T_X)$ to another $(Y,\mu_Y,T_Y)$ is a measure-preserving transformation $\pi \colon X \to Y$ such that $T_Y^h \circ \pi = \pi \circ T_X^h$ for all $h \in \Z^\omega$.
\end{itemize}
\end{definition}

One can show that $\| \cdot \|_{U^d(X)}$ is indeed a seminorm on $L^\infty(X)$; see for instance \cite[\S 2]{tz-concat}, where it is also shown that these seminorms agree with the ones defined in \cite{host2005nonconventional} using cubic measures.  
From \eqref{bo}, \eqref{bo-global} we observe the identity
\begin{equation}\label{bo-ident}
\langle (f_\omega)_{\omega \in \{0,1\}^{d+1}} \rangle_{U^{d+1}(X)}
= \lim_{n \to \infty} \E_{h^0,h^1 \in \Phi_n}
\langle (T^{h^0} f_{\omega,0} T^{h^1} \overline{f_{\omega,1}})_{\omega \in \{0,1\}^{d}} \rangle_{U^{d}(X)}
\end{equation}
for any $s \geq 0$ and $f_\omega \in L^\infty(X)$, $\omega \in \{0,1\}^{d+1}$; compare with \eqref{inner-recurse}. 

We can now give

\begin{theorem}[Conze--Lesigne inverse theorem for $\Z^\omega$-actions]\label{hk-inverse}  Let $(X,\mu,T)$ be an ergodic separable $\Z^\omega$-system, and let $f \in L^\infty(X,\mu)$ be such that $\|f\|_{U^{3}(X)} > 0$.  Then there exists a filtered locally compact group $H$ of degree at most $2$, a filtered lattice $\Gamma$ of $H$, a translation action $T_{H/\Gamma} \colon \Z^\omega \to \Aut(H/\Gamma)$ (equipping $H/\Gamma$ with Haar measure) defined by
\begin{equation}\label{thh}
T_{H/\Gamma}^h(x) \coloneqq \phi(h) x
\end{equation}
for some group homomorphism $\phi \colon \Z^\omega \to H$, a $\Z^\omega$-morphism $\Pi \colon X \to H/\Gamma$, and a continuous function $F \colon H/\Gamma \to \C$ such that
\begin{equation}\label{fFpi}
 \int_X f(x) \overline{F(\Pi(x))}\ d\mu \neq 0.
 \end{equation}
\end{theorem}

\begin{proof}  For the definition of any unexplained term in this argument we refer to \cite{jt21}.  The hypothesis  $\|f\|_{U^{3}(X)} > 0$ is equivalent to $\E(f|\mathrm{Z}^2(X)) \neq 0$, where $\mathrm{Z}^2(X)$ is the second Host--Kra--Ziegler factor; see e.g., \cite[Appendix A]{btz}.  By \cite[Lemma 1.2]{jt21}, $\mathrm{Z}^2(X)$ is a Conze--Lesigne system.  By \cite[Theorem 1.7]{jt21}, $\mathrm{Z}^2(X)$ is then the inverse limit of translational systems $H/\Gamma$, with $H = (H, H_2)$ a filtered locally compact Polish group of degree at most $2$, $\Gamma$ a filtered lattice in $H$, and with the action given by \eqref{thh}.  Thus we have $\E(f|H/\Gamma) \neq 0$ for at least one such system.  By Lusin's theorem this implies the existence of a continuous function $F \colon H/\Gamma \to \C$ such that
\begin{equation}\label{hgam}
 \int_{H/\Gamma} \E(f|H/\Gamma)(y) \overline{F(y)}\ d\mu_{H/\Gamma}(y) \neq 0
 \end{equation}
where $d\mu_{H/\Gamma}$ is the Haar probability measure on $H/\Gamma$.  As $H/\Gamma$ is a factor of $X$, there is an (abstract) factor map $\Pi \colon X \to H/\Gamma$, which can be upgraded to a concrete probability-preserving map by \cite[Proposition A.2(i)]{jt21}.  Pulling back \eqref{hgam} by $\Pi$ we obtain \eqref{fFpi}, and the claim follows.
\end{proof}

\begin{remark} Note in the above theorem that $H$ is merely locally compact rather than a Lie group.  This is necessary in order for the theorem to hold; see the example presented after \cite[Conjecture 2.14]{shalom1} (in the discussion of \cite[Theorem 4.3]{shalom1}).  In particular, we cannot necessarily take $H/\Gamma$ to be a nilmanifold, although we shall later see that it is still an inverse limit of nilmanifolds (in the category of compact nilspaces, not the category of $\Z^\omega$-systems).  On the other hand, the requirement that $X$ be separable can be easily removed in practice; see \cite{jt21} for further discussion.
\end{remark}

In Sections \ref{nonst-impl}-\ref{stab-sec} we establish the following connection between Theorem \ref{hk-inverse} and Conjecture \ref{uk-inverse-nil}:

\begin{theorem}[Conze--Lesigne inverse theorem implies Gowers inverse theorem]\label{implication} Theorem \ref{hk-inverse} implies the $s=2$ case of Conjecture \ref{uk-inverse-nil}.
\end{theorem}

Our arguments are an adaptation of those in \cite{tz-finite}, in which a random sampling method is used to obtain a correspondence principle that allows one to connect the combinatorial setting of Gowers uniformity to the ergodic setting of Host--Kra uniformity, and a rigidity result for polynomial maps is used to obtain the desired combinatorial conclusion.  Similar strategies (omitting the random sampling step as one no longer seeks to introduce dynamics) have also been employed in the nilspace framework \cite{szegedy-inv2}, \cite{szegedy-inv}, \cite{candela-szegedy-inverse}.  Our arguments are in fact valid for any $s$, in that the obvious generalization of Theorem \ref{hk-inverse} to a higher value of $s$ (replacing the $U^3$ seminorm with the $U^{s+1}$ norm, and making $H$ of degree $s$ rather than degree $2$) would imply the corresponding case of Conjecture \ref{uk-inverse-nil}.  However, we suspect that these higher degree analogues of Theorem \ref{hk-inverse} may be false as stated, and may need some further modification in order to be able to prove Conjecture \ref{uk-inverse-nil} for general $s$.  We hope to report further on this question in future work.

\subsection{Notation}\label{notation-sec}

We use $X \ll Y$, $Y \gg X$, $X = O(Y)$ to denote the estimate $|X| \leq C Y$ for some constant $C$, and write $X \asymp Y$ for $X \ll Y \ll X$.  If we need the implied constant $C$ to depend on additional parameters, we indicate this by subscripts, thus for instance $X \gg_{\eta,s} 1$ denotes an estimate of the form $X \geq c_{\eta,s}$ for some $c_{\eta,s}>0$ depending only on $\eta,s$.

In some sections we will switch to nonstandard asymptotic notation, in particular using $o(1)$ to refer to an infinitesimal quantity; see Appendix \ref{nonst} for details.

Given a measure space $X = (X, \X, \mu)$, we define the Lebesgue spaces $L^p(X) = L^p(X,\mu) = L^p(X,\X,\mu)$ for $1 \leq p \leq \infty$ in the usual fashion, in particular identifying any two functions that agree $\mu$-almost everywhere.

\section{Quadratic Fourier analysis arguments}\label{Fourier-sec}

\subsection{The structure of locally bilinear forms}

Our proof of Theorem \ref{inverse-u3} uses the same quadratic Fourier analysis arguments that were also used in \cite{green-tao-inverseu3}, \cite{samorodnitsky}.  The main new ingredient in addition to these arguments is a finer analysis of the structure of locally bilinear forms.  We first give a definition of these forms.

\begin{definition}  Let $U$ be a subset of an additive group $G$, and let $H$ be another additive group.  A \emph{locally bilinear form} $B \colon U \times U \to H$ is a function such that 
$$ B(x_1+x_2, y) = B(x_1,y) + B(x_2,y)$$
whenever $x_1,x_2,x_1+x_2,y \in U$, and similarly
$$ B(x, y_1+y_2) = B(x,y_1) + B(x,y_2)$$
whenever $x,y_1,y_2,y_1+y_2 \in U$.
If $U=G$, we refer to $B$ as a \emph{globally bilinear form}.  If $B(x,y)=B(y,x)$ for all $x,y \in U$, we call $B$ \emph{symmetric}.
\end{definition}

Consider a Bohr set $\Bohr(S,\rho)$ of an additive group $G$.  Clearly, any global bilinear form $\tilde B \colon G \times G \to H$ restricts to a local bilinear form $B \colon \Bohr(S,\rho) \times \Bohr(S,\rho) \to H$.  Unfortunately, the converse is not true; it is not difficult to find local bilinear forms $B \colon \Bohr(S,\rho) \times \Bohr(S,\rho) \to H$ that do not extend to a global bilinear form (e.g., see Example \ref{linear}).  However, it turns out that global extensions (or more precisely, local lifts) do exist as long as one also lifts the group $G$ to a larger group.  The construction is as follows.  If $G$ is a finite additive group and $S \subset \hat G$ is a set of frequencies, we let $(\xi)_{\xi \in S} \colon G \to (\R/\Z)^S$ denote the homomorphism
$$ (\xi)_{\xi \in S}(x) \coloneqq ( \xi \cdot x )_{\xi \in S}.$$
Thus for instance $\Bohr(S,\rho)$ is the inverse image of the cube $(-\rho,\rho)^S$ under this homomorphism $(\xi)_{\xi \in S}$, where we view $(-\rho,\rho)$ as a subset of $\R/\Z$.  We then define a new group $G_S \leq G \times \R^S$ by the formula
$$ G_S \coloneqq \{ (x, \theta) \in G \times \R^S \colon (\xi)_{\xi \in S}(x) = \theta \mod \Z^S \},$$
that is to say $G_S$ consists of tuples $(x, (\theta_\xi)_{\xi \in S})$ where $x \in G$ and for each $\xi \in S$, $\theta_\xi$ is a real number with $\theta_\xi = \xi \cdot x \mod 1$.  Clearly, $G_S$ is a lattice, and thus $G_S$ is a finitely generated abelian group, and that one has a short exact sequence 
\begin{equation}\label{exact}
0 \to \Z^S \to G_S \to G \to 0
\end{equation}
where the inclusion homomorphism from $\Z^S$ to $G_S$ is given by $\theta \mapsto (0,\theta)$, and the projection homomorphism from $G_S$ to $G$ is given by $(x,\theta) \mapsto x$.  Thus $G_S$ is an extension of $G$ by the lattice $\Z^S$.

If we endow $\R^S$ with the $\ell^\infty$ norm $\| \theta \|_{\R^S} \coloneqq \sup_{\xi \in S} |\theta_\xi|$ (with the convention that the supremum is identically zero when $S$ is empty) and define the cubes
$$ B_{\R^S}(0,\rho) \coloneqq \{ \theta \in \R^S: \|\theta\|_{\R^S} < \rho \} = (-\rho,\rho)^S$$
then by chasing definitions, we see that we have an identification
\begin{equation}\label{bohr}
 G_S \cap (G \times B_{\R^S}(0,\rho)) \equiv \Bohr(S,\rho)
\end{equation}
for $0 < \rho < 1/2$, where the identification map is given by restricting the projection homomorphism $\pi \colon G_S \to G$ to the set $G_S \cap (G \times B_{\R^S}(0,\rho))$.

The following useful lemma lets us (locally) lift a locally bilinear form to a globally bilinear one:

\begin{lemma}[Globalising a locally bilinear form]\label{ext}  Let $G$ be a finite additive group, let $S \subset \hat G$ be a set of frequencies, and let $0 < \rho < 1/2$.  Suppose that we have a locally bilinear form $B \colon \Bohr(S,\rho) \times \Bohr(S,\rho) \to H$ into some additive group $H$.  Then there exists a globally bilinear form $\tilde B \colon G_S \times G_S \to H$ such that
\begin{equation}\label{tbb}
 \tilde B( (x, \theta), (y,\sigma) ) = B(x,y)
 \end{equation}
whenever $(x,\theta), (y,\sigma) \in G_S$ are such that $\|\theta\|_{\R^S}, \|\sigma\|_{\R^S} \leq \exp( - C |S|^C ) \rho$ for some absolute constant $C>0$.  Furthermore, if $B$ is symmetric, then one can choose $\tilde B$ to also be symmetric.
\end{lemma}

\begin{example}\label{linear}  Let $G = \Z/N\Z$ be a cyclic group, let $\alpha$ be a real number, and let $0 < \rho < 1/4$. 
Let $\phi:\Z/N\Z\to \Z$ be the lift $\phi(x=[n]) = n \mod N$, i.e., a group homomorphism such that $\pi\circ \phi=\id_{\Z/N\Z}$, where $\pi:\Z\to \Z/N\Z$ is the canonical projection. 
Let $\frac{1}{N} \colon \Z \to \R/\Z$ be the operation of division by $N$. 
Let $S=\{\psi=\frac{1}{N}\circ \phi\}$ and $\Bohr(S,\rho) = \{ x\in \Z/N\Z: \|\psi(x)\|_{\R/\Z} < \rho \}$.  
The form $B \colon \Bohr(S,\rho) \times \Bohr(S,\rho) \to \R$ defined by
$$ B( [n], [m]) \coloneqq \alpha \phi(n)\phi(m)$$
is locally bilinear, but will not in general be extendable to a global bilinear form on $\Z/N\Z$.  
The problem will be that  we may have carry-over effects on $\Z/N\Z$ which would destroy linearity of $B$.   
However, on the infinite cyclic subgroup
$$ G_S \coloneqq \left\{ (x,\theta) \in \Z/N\Z \times \R: \frac{x}{N} = \theta \mod 1 \right\} = \left\{ \left(n \mod N, \frac{n}{N}\right): n \in \Z \right\}$$
of $\Z/N\Z \times \R$, we can define the global bilinear form $\tilde B \colon G_S \times G_S \to \R$ by
$$ \tilde B\left( \left(n \mod N, \frac{n}{N}\right), \left(m \mod N, \frac{m}{N}\right) \right) \coloneqq \alpha nm$$
and then it is clear that \eqref{tbb} holds for all $(x,\theta), (y,\sigma) \in G_S$ with $|\theta|, |\sigma| < \rho$. 
By creating a second coordinate, we could honestly record the effects of carry-over.   
\end{example}

\begin{proof}  We modify the proof of \cite[Lemma 4.22]{tao-vu}.  We may assume that $|S| \geq 1$, since the claim is trivial for $|S|=0$.  Using the equivalence \eqref{bohr}, we have a locally bilinear form 
$$ B' \colon \left( G_S \cap (G \times B_{\R^S}(0,\rho))\right) \times \left(G_S \cap (G \times B_{\R^S}(0,\rho))\right) \to H$$
such that
$$ B'( (x, \theta), (y,\sigma) ) = B(x,y)$$
whenever $(x,\theta), (y,\sigma) \in G_S$ are such that $\|\theta\|_{\R^S}, \|\sigma\|_{\R^S} \leq \rho$, with $B'$ symmetric whenever $B$ is.  Our task is thus to locate a globally bilinear form $\tilde B \colon G_S \times G_S \to H$ such that
\begin{equation}\label{tbx}
 \tilde B( (x,\theta), (y,\sigma) ) = B'( (x, \theta), (y,\sigma) ) 
 \end{equation}
whenever $\|\theta\|_{\R^S}, \|\sigma\|_{\R^S} \leq \exp(-C |S|^C ) \rho$, with $\tilde B$ symmetric whenever $B'$ is.

Let $\Gamma \coloneqq \{ \theta \colon (x,\theta) \in G_S \}$, then $\Gamma$ is a subgroup of $\R^S$ that contains $\Z^S$ as a finite index subgroup, and is thus a lattice.  Applying the discrete John's theorem \cite[Lemma 1.6]{tao-vu-john} (see also \cite{berg-henk} for recent refinements), we can find linearly independent vectors $w_\xi \in \Gamma$ and real numbers $N_\xi > 0$ for $\xi \in S$ such that
$$ B_{\R^S}\left(O(|S|)^{-3|S|/2} t\right) \cap \Gamma \subset (-tN,tN) \cdot w \subset B_{\R^S}(t) \cap \Gamma $$
for any $t>0$, where
\begin{equation}\label{tnw}
 (-tN,tN) \cdot w \coloneqq \{ n_1 w_1 + \dots + n_{|S|} w_{|S|}: n_1,\dots,n_{|S|} \in \Z; |n_i| < t N_i \hbox{ for all } i=1,\dots,|S|\}.
 \end{equation}
In other words, the $w_1,\dots,w_{|S|}$ generate $\Gamma$, and one has the inequalities
$$ O(|S|)^{-3|S|/2} \sup_{1 \leq i \leq |S|} \frac{|n_i|}{N_i} \leq \| n_1 w_1 + \dots + n_{|S|} w_{|S|}\|_{\R^S} \leq \sup_{1 \leq i \leq |S|} \frac{|n_i|}{N_i}$$
for all integers $n_1,\dots,n_{|S|}$.

We may relabel so that $N_1 \geq \dots \geq N_{|S|}$.  We can find $v_1,\dots,v_{|S|} \in G_S$ such that each $v_i$ is of the form $v_i = (g_i,w_i)$ for some $g_i \in G$.  Every element $(x,\theta)$ of $G_S$ can thus be uniquely represented in the form
$$ (x,\theta) = (y,0) + n_1 v_1 + \dots + n_{|S|} v_{|S|}$$
where $n_1,\dots,n_{|S|}$ are integers and $y \in K \coloneqq \{ y \in G: (y,0) \in G_S\}$.  If we let $j$ be the largest index for which $N_j > 1$ (or $j=0$ if no such $N_j$ exists), we conclude that the form
$$ ((y,n_1,\dots,n_j), (y',n'_1,\dots,n'_j)) \mapsto
B'( (y,0) + n_1 v_1 + \dots + n_j v_j, (y',0) + n'_1 v_1 + \dots + n'_j v_j )$$
is a locally bilinear form on
$$ \{ (y,n_1,\dots,n_j) \in K \times \Z^j: |n_i| < N_i \ \forall 1 \leq i \leq j \}.$$
In particular, this form can be expanded in coordinates as
$$ B'((y,0),(y',0)) + \sum_{i=1}^j n_i B'(v_i,(y',0)) + n'_i B'((y,0),v'_i) + \sum_{1 \leq i,i' \leq j} n_i n_{i'} B'(v_i,v_{i'}).$$
This latter expression is well-defined in $H$ for all $y,y' \in K$ and $v_1,\dots,v_{|S|}, v'_1,\dots,v'_{|S|} \in \Z$ (those $v_i, v'_i$ with $i > j$ play no role in this form).  We can thus extend this form to a globally bilinear form on $K \times \Z^{|S|}$, which we can identify with $G_S$; by inspection we see that this form is symmetric if $B'$ is.  From \eqref{tnw} we see that \eqref{tbx} holds whenever 
$\|\theta\|_{\R^S}, \|\sigma\|_{\R^S} \leq (C|S|)^{-3|S|/2} \rho$ for some absolute constant $C$, and the claim follows.
\end{proof}

As our main application of this globalisation lemma we can express symmetric bilinear forms in terms of quadratic polynomials.

\begin{lemma}[Integrating a symmetric bilinear form]\label{integ}  Let $G$ be a finitely generated additive group. 
\begin{itemize}
\item[(i)] If $B \colon G \times G \to \R/\Z$ is a symmetric globally bilinear form, then there exists a globally quadratic function $\phi \colon G \to \R/\Z$ such that $\phi(x+y) = \phi(x)+\phi(y) + B(x,y)$ for all $x,y \in G$.
\item[(ii)]  If $G$ is finite, $S \subset \hat G$, $\rho>0$, and $B: \Bohr(S,\rho) \times \Bohr(S,\rho) \to \R/\Z$ is a symmetric locally bilinear form, then there exists $\exp(-|S|^{O(1)}) \rho \ll \rho' \leq \rho$ and a locally quadratic function $\phi \colon \Bohr(S,\rho') \to \R/\Z$ such that
$\phi(x+y) = \phi(x)+\phi(y) + B(x,y)$ for all $x,y \in \Bohr(S,\rho'/2)$.  Furthermore, there is a globally quadratic lift $\tilde \phi \colon G_S \to \R/\Z$ such that $\tilde \phi(x,\theta) = \phi(x)$ whenever $(x,\theta) \in G_S$ and $\|\theta\|_{\R^S} < \rho'$.
\end{itemize}
\end{lemma}

It is possible that the finite generation or finiteness hypotheses on $G$ here can be relaxed, but we will not need to do so here. If $|G|$ were odd, then one proceed by setting $\phi(x) \coloneqq B(x, \frac{1}{2} x)$ since one now has the ability to divide by two in $G$; this observation was implicit in the arguments in \cite{green-tao-inverseu3}.  Similarly, if $G = \F_2^N$, then (for part (i) at least) one could proceed by inspecting the matrix coefficients of $B$ and extracting the upper triangular half to define $\phi$; this observation was implicit in the arguments in \cite{samorodnitsky}.  The fact that the above lemma holds for arbitrary finite additive groups $G$ is the main reason why we are able to establish Theorem \ref{inverse-u3} without additional hypotheses on $G$.

\begin{proof}  We begin with (i).  We first observe that if the claim is true for two finitely generated abelian groups $G_1,G_2$ then it is true for the direct product $G_1 \times G_2$.  Indeed, if $B: (G_1 \times G_2) \times (G_1 \times G_2) \to \R/\Z$ is a symmetric bilinear form and the claim is already established for $G_1,G_2$, then one can find quadratic functions $\phi_1: G_1 \to \R/\Z$, $\phi_2: G_2 \to \R/\Z$ such that
$$ \phi_1(x_1+y_1) = \phi_1(x_1) + \phi_1(y_1) + B((x_1,0),(y_1,0))$$
and
$$ \phi_2(x_2+y_2) = \phi_2(x_2) + \phi_2(y_2) + B((0,x_2),(0,y_2))$$
for all $x_1,y_1 \in G_1$ and $x_2,y_2 \in G_2$.  One then checks from direct calculation (using the symmetric bilinear nature of $B$) that the function $\phi: G_1 \times G_2 \to \R/\Z$ defined by
$$ \phi(x_1,x_2) := \phi_1(x_1) + \phi_2(x_2) + B((x_1,0),(0,x_2))$$
is a quadratic form on $G_1 \times G_2$ that obeys the required property 
$$ \phi(x_1+y_1,x_2+y_2) = \phi(x_1,y_1) + \phi(x_2,y_2) + B( (x_1,x_2), (y_1,y_2) ).$$
Since every finitely generated abelian group is the direct product of finitely many copies of the integers $\Z$ or cyclic groups $\Z/N\Z$, it thus suffices to verify the claim when $G$ is either the integers or a cyclic group.  If $G=\Z$, then the bilinear form $B$ takes the form
$$ B(x,y) = \alpha xy \mod 1$$
for some real number $\alpha \in \R$ (cf. Example \ref{linear}), and one can take $\phi(x) \coloneqq \frac{\alpha x^2}{2} \mod 1$.  If $G = \Z/N\Z$, then the bilinear form $B$ similarly takes the form
$$ B(x,y) = \frac{axy}{N} \mod 1$$
for some integer $a$, and then one can take
$$ \phi(x \mod N) \coloneqq \frac{a\binom{x}{2}}{N} - \frac{a x \binom{N}{2}}{N^2} \mod 1$$
for all integers $N$, observing from the identity $\binom{x+N}{2} = \binom{x}{2} + \binom{N}{2} + Nx$ that the right-hand side is $N$-periodic (regardless of whether $N$ is even or odd).  This establishes (i).

Now we establish (ii).  By Lemma \ref{ext}, we can find a symmetric globally bilinear form $\tilde B \colon G_S \times G_S \to \R/\Z$ such that
\begin{equation}\label{b1}
 \tilde B( (x, \theta), (y,\sigma) ) = B(x,y)
 \end{equation}
whenever $(x,\theta), (y,\sigma) \in G_S$ are such that $\|\theta\|_{\R^S}, \|\sigma\|_{\R^S} \leq \exp( - C |S|^C ) \rho$ for some absolute constant $C>0$.  By part (i), we can then find a globally quadratic function $\tilde \phi \colon G_S \to \R/\Z$ such that
\begin{equation}\label{b2}
 \tilde \phi(x+y, \theta+\sigma) = \tilde \phi(x,\theta)+\tilde \phi(y,\sigma) + B((x,\theta),(y,\sigma))
 \end{equation}
for any $(x,\theta), (y,\sigma) \in G_S$.  Using \eqref{bohr}, we can now define $\phi \colon \Bohr(S,\rho) \to \R/\Z$ so that 
\begin{equation}\label{b3} 
\tilde \phi(x,\theta) = \phi(x)
\end{equation}
whenever $(x,\theta) \in G_S$ are such that $\|\theta\|_{\R^S} \leq \rho$.  In particular, $\phi$ is locally quadratic on $\Bohr(S,\rho)$. Combining \eqref{b1}, \eqref{b2}, \eqref{b3}, we obtain the claim (setting $\rho' \coloneqq \frac{1}{100} \exp(-C |S|^C) \rho$, say).
\end{proof}

\subsection{Proof of inverse theorem}  

We can now prove Theorem \ref{inverse-u3}.  For the rest of the section, we let $G, f, \eta$ be as in the hypotheses of that theorem.  In this section, we adopt the convention from  \cite{green-tao-inverseu3} of using $b()$ to denote various $1$-bounded functions of the indicated variables, with the functions allowed to vary from occurrence to occurrence.

We now execute some standard arguments from \cite{gowers1}, \cite{green-tao-inverseu3}, \cite{samorodnitsky} in which we obtain an increasing amount of control on various objects relating to $f$.  
We begin with attaching a frequency $Mh$ depending in a locally linear fashion on $h$ to all $h$ in a Bohr set:

\begin{lemma}\label{part-2}  There exists a regular Bohr set $B(S,\rho)$ with $|S| \ll \eta^{-O(1)}$ and $1 \ll \rho \leq 1/8$ and a locally linear function $M \colon \Bohr(S,2\rho) \to \hat G$ such that
$$ \left|\E_{\subalign{x &\in G \\h &\in \Bohr(S,\rho)}} b(h) b(x) f(x+h) e(-Mh \cdot x)\right| \gg \eta^{O(1)}.$$
\end{lemma}

\begin{proof} This is a slight modification of \cite[Proposition 9.3]{green-tao-inverseu3}, with some factors of $2$ removed to account for the fact that $|G|$ is no longer required\footnote{This hypothesis was omitted by mistake in the hypotheses of that proposition, as well as in \cite[Proposition 9.1]{green-tao-inverseu3}.} to be odd.  More specifically, one repeats the arguments in \cite[Proposition 9.1]{green-tao-inverseu3}, until one reaches the conclusion that $2\Gamma''-2\Gamma''$ is a graph.  In \cite{green-tao-inverseu3}, the (implicit) assumption that $|G| = |\hat G|$ was odd was used to write the portion of this graph over $B_0$ as\footnote{There is a typo at this step of the proof in \cite{green-tao-inverseu3}, in that $2H''-2H''$ should be $2\Gamma''-2\Gamma''$.} $\{ (h, 2M(h)): h \in B_0 \}$ for some function $M \colon B_0 \to \hat G$.  But for general values of $|G|$, we can instead write this portion as $\{ (h, M(h)): h \in B_0 \}$.  If one then continues the arguments of \cite[Proposition 9.1]{green-tao-inverseu3} and \cite[Proposition 9.3]{green-tao-inverseu3}, replacing all occurrences of $2M$ by $M$, one obtains the claim.  We remark that these arguments rely heavily on tools from additive combinatorics, such as the Balog--Szemer\'edi--Gowers lemma, the Pl\"unnecke inequalities, and the Bogulybov lemma.
\end{proof}

Next, we replace the locally bilinear form $(h,x) \mapsto Mh \cdot x$ with a symmetric locally bilinear form:

\begin{lemma}\label{part-3}  There exists a regular Bohr set $B(S,\rho)$ with $|S| \ll \eta^{-O(1)}$ and $\eta^{O(1)} \ll \rho \leq 1/8$ and a symmetric locally bilinear form $B \colon \Bohr(S,2\rho) \times \Bohr(S,2\rho) \to \R/\Z$ such that
$$ \left|\E_{\subalign{x &\in G\\ h,k &\in \Bohr(S,\rho)}} b(x,h) b(x,k) f(x+h+k) e(-B(h,k))\right| \gg \eta^{O(1)}.$$
\end{lemma}

\begin{proof} Let $M$ be the locally linear function from Lemma \ref{part-2}.  We follow the ``symmetry argument'' used to establish \cite[Lemma 9.4]{green-tao-inverseu3}, but without the additional factor of $2$ (and in particular deleting the arguments in the final paragraph devoted to removing this power of $2$) to conclude (after adjusting the Bohr set $\Bohr(S,\rho)$ from Lemma \ref{part-2} appropriately) that
$$ \| M x \cdot z - Mz \cdot x \|_{\R/\Z} \ll \eta^{-O(1)} \| x\|_S$$
whenever $x,z \in \Bohr(S,\rho)$.  In particular, if $\rho' \sim C^{-1} \eta^C \rho$ for a sufficiently large absolute constant $C$ with $\Bohr(S,\rho')$ regular (the existence of which is guaranteed by \cite[Lemma 8.2]{green-tao-inverseu3}), one has
$$ \| M x \cdot z - Mz \cdot x \|_{\R/\Z} < \eta^{C/2} < \frac{1}{10}$$
(say) for $x,z \in \Bohr(S,\rho')$.  We can thus uniquely define a ``midpoint'' $B(x,z) \in \R/\Z$ and a difference $\Delta(x,z) \in [-\eta^C, \eta^C] \subset [-1/10,1/10]$ for $x,z \in \Bohr(S,\rho')$ such that
\begin{equation}\label{max}
M x \cdot z = B(x,z) + \Delta(x,z); \quad M z \cdot x = B(x,z) - \Delta(x,z)
\end{equation}
for $x,z \in \Bohr(S,\rho')$.  One then easily verifies that $B$ is a symmetric locally bilinear form on $\Bohr(S,\rho')$.  From Lemma \ref{part-2} and a change of variables (using \cite[Lemma 4.2(ii)]{green-tao-inverseu3}) we have (for $C$ large enough)
$$ \left|\E_{\subalign{x &\in G \\ h &\in \Bohr(S,\rho)\\ y,z &\in \Bohr(S,\rho')}} b(h+y) b(x+z) f(x+z+h+y) e(-M(h+y) \cdot (x+z))\right| \gg \eta^{O(1)}$$
which after collecting terms can be rearranged as
$$ \left|\E_{\subalign{x &\in G\\ h &\in \Bohr(S,\rho)\\ y,z &\in \Bohr(S,\rho')}} b(x,h,y) b(x,h,z) f(x+h+y+z) e(-My \cdot z)\right| \gg \eta^{O(1)};$$
using \eqref{max} to approximate $My \cdot z$ by $B(y,z)$ and using the triangle inequality, we obtain (for $C$ large enough)
$$ \left|\E_{\subalign{x &\in G\\ h &\in \Bohr(S,\rho)\\ y,z &\in \Bohr(S,\rho')}} b(x,h,y) b(x,h,z) f(x+h+y+z) e(-B(y,z))\right| \gg \eta^{O(1)}.$$
Making the substitution $x' = x+h$ and then pigeonholing in $h$ we conclude that
$$ \left|\E_{\subalign{x' &\in G\\ y,z &\in \Bohr(S,\rho')}} b(x',y) b(x',z) f(x'+y+z) e(-B(y,z))\right| \gg \eta^{O(1)}$$
and the claim follows after a relabeling.
\end{proof}

Now we can prove Theorem \ref{inverse-u3}.  Let $S,\rho$ be as in Lemma \ref{part-3}.  By Lemma \ref{integ} (and \cite[Lemma 8.2]{green-tao-inverseu3}), we can find $\exp(-\eta^{-O(1)}) \rho \ll \rho_1 < \rho/100$ with $\Bohr(S,\rho_1)$ regular, and a locally quadratic function $\phi \colon \Bohr(S,2\rho_1) \to \R/\Z$ such that
\begin{equation}\label{phi-b}
\phi(x+y) = \phi(x)+\phi(y) + B(x,y)
\end{equation}
for all $x,y \in \Bohr(S,\rho_1)$, as well as a globally quadratic lift $\tilde \phi \colon G_S \to \R/\Z$ such that $\tilde \phi(x,\theta) = \phi(x)$ whenever $(x,\theta) \in G_S$ and $\|\theta\|_{\R^S} < 10 \rho_1$.

Let $C$ be a sufficiently large constant, then by another application of \cite[Lemma 8.2]{green-tao-inverseu3} we may find  $\rho_2 \asymp C^{-1} \eta^C \rho_1$ such that $\Bohr(S,\rho_2)$ is regular. From Lemma \ref{part-3} and a change of variables (using \cite[Lemma 4.2(ii)]{green-tao-inverseu3}) we have
$$ \left|\E_{\subalign{x &\in G\\ h,k &\in \Bohr(S,\rho)\\ y &\in \Bohr(S,\rho_1)\\ z &\in \Bohr(S,\rho_2)}} b(x,h+y+z) b(x,k-y) f(x+h+k+z) e(-B(h+y+z,k-y))\right| \gg \eta^{O(1)}.$$
From the locally bilinear nature of $B$, followed by \eqref{phi-b}, we may write
\begin{align*}
 B(h+y+z,k-y) &= B(h,k) - B(h,y) + B(y+z,k) - B(y,y) - B(y,z) \\
 &= B(h,k) - B(h,y) + B(y+z,k) - B(y,y) - \phi(y+z) + \phi(y) + \phi(z)
\end{align*}
so on collecting terms we obtain
$$ \left|\E_{\subalign{x &\in G\\ h,k &\in \Bohr(S,\rho)\\ y &\in \Bohr(S,\rho_1)\\ z &\in \Bohr(S,\rho_2)}} b(x,h,k,y+z) b(x,h,k,y) f(x+h+k+z) e(-\phi(z))\right| \gg \eta^{O(1)}.$$
Making the change of variables $x' \coloneqq x+h+k$ and pigeonholing in $h,k$, we conclude that
$$ \left|\E_{\subalign{x' &\in G\\ y &\in \Bohr(S,\rho_1)\\ z &\in \Bohr(S,\rho_2)}} b(x',y+z) b(x',y) f(x'+z) e(-\phi(z))\right| \gg \eta^{O(1)}.$$
Thus, we have
$$ \left|\E_{\subalign{y &\in \Bohr(S,\rho_1)\\ z &\in \Bohr(S,\rho_2)}} b(x',y+z) b(x',y) f(x'+z) e(-\phi(z))\right| \gg \eta^{O(1)}$$
for $\gg \eta^{O(1)} |G|$ values of $x' \in G$.  Applying  \cite[Lemma 4.4]{green-tao-inverseu3}, we conclude for each such $x'$ that there exists $\xi(x') \in \hat G$ such that
$$ \left|\E_{z \in \Bohr(S,\rho_1)} f(x'+z) e(-\phi(z)-\xi(x') \cdot z)\right| \gg \eta^{O(1)}$$
and Theorem \ref{inverse-u3} follows.

For future reference we observe that in fact we have demonstrated the following slight refinement of Theorem \ref{inverse-u3}:

\begin{theorem}\label{inverse-u3-plus} In the conclusion of Theorem \ref{inverse-u3}, one can also ensure that there exists a globally quadratic lift $\tilde \phi \colon G_S \to \R/\Z$ of $\phi$ such that $\tilde \phi(x,\theta) = \phi(x)$ whenever $(x,\theta) \in G_S$ and $\|\theta\|_{\R^S} < 10 \rho$.
\end{theorem}

\section{From locally quadratic functions to nilmanifolds}\label{tonil-sec}

We now prove Theorem \ref{inverse-u3-nil}.  The algebraic constructions here share some similarity with those in \cite{manners-periodic}, in particular taking advantage of the semidirect product construction to convert the outer automorphisms arising from a translation action into inner automorphisms.

We turn to the details. By Theorem \ref{inverse-u3} and the pigeonhole principle, there exists a regular Bohr set $\Bohr(S, \rho)$ with $|S| \ll \eta^{-O(1)}$ and $\exp(-\eta^{-O(1)}) \ll \rho < 1/2$ and $x_0 \in G$ such that
$$ |\E_{h \in \Bohr(S,\rho)} f(x_0+h) e(-\phi(h))| \gg \eta^{O(1)}$$
for some locally quadratic $\phi$, which can be lifted to a globally quadratic function $\tilde \phi \colon G_S \to \R/\Z$ in the sense that $\tilde \phi(x,\theta) = \phi(x)$ whenever $(x,\theta) \in G_S$ and $\|\theta\|_{\R^S} < 10 \rho$.  By \eqref{bohr} we then have
$$\left |\sum_{(h,\theta) \in G_S} f(x_0+h) 1_{\|\theta\|_{\R^S} < \rho } e(-\tilde \phi(h,\theta))\right| \gg \exp(-\eta^{-O(1)}) |G|.$$
Smoothing out the cutoff $1_{\|\theta\|_{\R^S} < \rho }$ at a scale $\exp(-C \eta^{-C})$ for some sufficiently large absolute constant $C$, and using the regularity of the Bohr set, we conclude that
\begin{equation}\label{chao}
\left |\sum_{(h,\theta) \in G_S} f(x_0+h) \varphi(\theta) e(-\tilde \phi(h,\theta))\right| \gg \exp(-\eta^{-O(1)}) |G|
 \end{equation}
for some Lipschitz function $\varphi: \R^S \to \R$ supported on (say) the unit ball with Lipschitz norm at most $\exp(\eta^{-O(1)})$.

For any $(h,\theta) \in G_S$, the map $n \mapsto \tilde \phi(h,\theta+n)$ is a quadratic map from $\Z^S$ to $\R/\Z$. By performing a Taylor expansion of this map and then lifting the coefficients from $\R/\Z$ to $\R$, we may then find a quadratic map $\Phi(h,\theta) \colon \R^S \to \R$ with the property that
\begin{equation}\label{tphn}
\tilde \phi(h,\theta+n) = \Phi(h,\theta)(n) \mod 1
\end{equation}
for all $n \in \Z^S$.  This map is well defined up to an element of $\Poly_{\leq 2}(\Z^S \to \Z)$, and so by abuse of notation one can view $\Phi(h,\theta)$ as lying in the quotient group $\Poly_{\leq 2}(\R^S \to \R) / \Poly_{\leq 2}(\Z^S \to \Z)$.  From the globally quadratic nature of $\tilde \phi$, we see (using \eqref{relate}) that the map $\Phi \colon G_S \to \Poly_{\leq 2}(\R^S \to \R) / \Poly_{\leq 2}(\Z^S \to \Z)$ is a polynomial map, using the the abelian filtration on $G_S$ and the degree $2$ prefiltration 
\begin{align*}
\Poly_{\leq 2}(\R^S \to \R) / \Poly_{\leq 2}(\Z^S \to \Z)
&\geq \Poly_{\leq 1}(\R^S \to \R) / \Poly_{\leq 1}(\Z^S \to \Z) \\
&\geq \Poly_{\leq 0}(\R^S \to \R) / \Poly_{\leq 0}(\Z^S \to \Z) \\
& \geq 0 \geq 0 \dots
\end{align*}
on $\Poly_{\leq 2}(\R^S \to \R) / \Poly_{\leq 2}(\Z^S \to \Z)$.
We then define the map $g \colon G_S \to H(\R^S) / H(\Z^S)$ by defining
$$ g(h,\theta) \coloneqq ( \theta, \Phi(h,\theta) ) \mod H(\Z^S);$$
this is well-defined since $\Phi(h,\theta)$ can be viewed as an element of $\Poly_{\leq 2}(\R^S \to \R)$ modulo an element of $\Poly_{\leq 2}(\Z^S \to \Z)$, and $0 \times \Poly_{\leq 2}(\Z^S \to \Z)$ is a subgroup of $H(\Z^S)$.  For any $n\in \Z^S$, we have
\begin{align*}
 g(h,\theta+n) &= (\theta+n, \Phi(h,\theta+n)) \mod H(\Z^S) \\
 &= (\theta, T^{-n} \Phi(h,\theta+n) ) (n, 0) \mod H(\Z^S)\\
 &= (\theta, T^{-n} \Phi(h,\theta+n) )  \mod H(\Z^S)\\
 &= (\theta, \Phi(h,\theta) )  \mod H(\Z^S)\\
 &= g(h,\theta)
\end{align*}
where we use the fact (from \eqref{tphn}) that $T^{-n} \Phi(h,\theta+n)$ and $\Phi(h,\theta)$ differ by an element of $\Poly_{\leq 2}(\Z^S \to \Z)$.  Thus by the short exact sequence \eqref{exact}, $g$ actually descends to a map $g \colon G \to H(\R^S)/H(\Z^S)$, where we define $H(\R^S), H(\Z^S)$ as in Definition \ref{spef} but using $S$ as the index set instead of $\{1,\dots,N\}$.  If we define the map $F \colon H(\R^S) \to \C$ by
$$ F( \theta,\psi ) \coloneqq \sum_{n \in \Z^S} \varphi(\theta+n) e(\psi(n))$$
for $\theta \in \R^S$ and $\psi \in \Poly_{\leq 2}(\R^S \to \R)$, one can check that
$$ F((\theta,\psi)(m,\lambda)) = F(\theta,\psi)$$
for all $m \in \Z^S$ and $\lambda  \in \Poly_{\leq 2}(\Z^S \to \Z)$, so that $F$ descends to a function on $H(\R^S)/H(\Z^S)$.  From chasing the definitions we see that for every $(h,\theta) \in G_S$ we have
\begin{align*}
F(g(h)) &= F(\theta_0,\Phi(h,\theta_0))\\
&= \sum_{n \in \Z^S} \varphi(\theta_0+n) e(\Phi(h,\theta_0)(n))\\
&= \sum_{n \in \Z^S} \varphi(\theta_0+n) e(\tilde \phi(h,\theta_0+n)) \\
&= \sum_{\theta: (h,\theta) \in G_S} \varphi(\theta) e(\tilde \phi(h,\theta))
\end{align*}
whenever $(h,\theta_0) \in G_S$.  From \eqref{chao} we conclude that
$$
|\E_{h \in G} f(x_0+h) \overline{F}(g(h))| \gg \exp(-\eta^{-O(1)}). 
$$
It is not difficult to verify that $F$ is Lipschitz of norm at most $\exp(\eta^{-O(1)})$.  Since $H(\R^S)/H(\Z^S)$ is isomorphic to $H(\R^N)/H(\Z^N)$ with $N = |S|$, the only remaining task is to show that $g$ is a polynomial map from $G$ to $H(\R^S)/H(\Z^S)$.  As $G_S$ is a finitely generated abelian group, we may write $G_S$ (non-uniquely) as $G_S = \Pi(\Z^D)$ for some lattice $\Z^D$ and surjective homomorphism $\Pi$.  Then the map $\Phi \circ \Pi$ is a polynomial map from $\Z^D$ (with the abelian filtration) to $\Poly_{\leq 2}(\R^S \to \R)/ \Poly_{\leq 2}(\Z^S \to \Z)$ and hence (by lifting Taylor coefficients) can be lifted to a polynomial map $\Psi \colon \Z^D \to \Poly_{\leq 2}(\R^S \to \R)$, and then $(0,\Psi) \colon \Z^D \to H(\R^S)$ is also a polynomial map.  Meanwhile, the map
$$ n \mapsto (\pi_\theta(\Pi(n)), 0)$$
is also easily verified to be a polynomial map, where $\pi_\theta \colon G_S \to \R^S$ is the coordinate projection $\pi_\theta(h,\theta) \coloneqq \theta$.  Multiplying these two maps together (using the Lazard--Leibman theorem, see e.g., \cite[Corollary B.4]{gtz}), we conclude that the map $\tilde g \colon \Z^D \to H(\R^S)$ defined by
$$ \tilde g(n) \coloneqq (\pi_\theta(\Pi(n)), \Psi(n))$$
is also a polynomial map.  On the other hand, by comparing definitions we see that
$$ \tilde g(n) = g(\Pi(n)) \mod H(\Z^S)$$
for all $n \in \Z^D$.  Since $\Pi \colon \Z^D \to G_S$ is surjective, we conclude that $g \colon G_S \to H(\R^S)/H(\Z^S)$ is a polynomial map, and hence on descending to $G$ we conclude that $g \colon G \to H(\R^S)/H(\Z^S)$ as well, as desired.  This concludes the proof of Theorem \ref{inverse-u3-nil}.

\section{Nonstandard formulation}\label{nonst-impl}

We now begin the proof of Theorem \ref{implication}.  As is common in several other approaches to the inverse conjecture (e.g., \cite{gtz}, \cite{candela-szegedy-inverse}), the first step is to translate Conjecture \ref{uk-inverse-nil}  into the language of nonstandard analysis, in order to access tools such as Loeb measure.  The nonstandard analysis formalism we require is summarized in Appendix \ref{nonst} for convenience.

Let $G = \prod_{\n \to p} G_\n$ be a nonstandard finite additive group, that is to say an ultraproduct of standard finite additive groups $G_\n$.   As discussed in \cite[\S 5]{tz-concat}, the standard Gowers uniformity structures on $G_\n$ then induce analogous (external) structures on $G$; in particular, there is an inner product $\langle (f_\omega)_{\omega \in \{0,1\}^d} \rangle_{U^d(G)} \in \C$ one can define for any $d \geq 1$ and any (external) functions $f_\omega \in L^\infty(G)$, with the property that whenever the internal functions $f_\omega \coloneqq \lim_{\n \to p} f_{\omega,\n}$ are ultralimits of uniformly bounded standard functions $f_{\omega,\n} \colon G_\n \to \C$, one has
\begin{equation}\label{flim}
 \langle (\st f_\omega)_{\omega \in \{0,1\}^d} \rangle_{U^d(G)} = \st  \langle (f_\omega)_{\omega \in \{0,1\}^d} \rangle_{{}^* U^d(G)}
 \end{equation}
where the \emph{internal Gowers inner product} $\langle (f_\omega)_{\omega \in \{0,1\}^d} \rangle_{{}^* U^d(G)}$ of the $f_\omega$ is defined to be the nonstandard complex number
$$  \langle (f_\omega)_{\omega \in \{0,1\}^d} \rangle_{{}^* U^d(G)} \coloneqq \lim_{\n \to p}
 \langle (f_{\omega,\n})_{\omega \in \{0,1\}^d} \rangle_{U^d(G_\n)}.$$ 
In particular, we can define an (external) seminorm $\| \cdot \|_{U^d(G)}$ on $L^\infty(G)$ by the formula
$$ \|f\|_{U^d(G)} \coloneqq  \langle (f)_{\omega \in \{0,1\}^d} \rangle_{U^d(G)}^{1/2^d};$$
one can verify from H\"older's inequality that these norms and inner products are well-defined, with
$$  |\langle (f_\omega)_{\omega \in \{0,1\}^d} \rangle_{U^d(G)}| \leq \prod_{\omega \in \{0,1\}^d} \|f_\omega \|_{L^{2^d}(G)}$$
and
$$ \|f\|_{U^d(G)} \leq \|f\|_{L^{2^d}(G)}.$$

We can now formulate

\begin{conjecture}[Inverse theorem for $U^{s+1}(G)$, nonstandard formulation]\label{uk-inverse-nil-nonst}  Let $G$ be a nonstandard finite additive group (with the Loeb probability space structure), let $s \geq 1$ be standard, and let $f \in L^\infty(G)$ be such that $\|f\|_{U^{s+1}(G)} > 0$.  Then there exists a standard degree $s$ filtered nilmanifold $H/\Gamma$, a standard continuous function $F \colon H/\Gamma \to \C$, and an internal polynomial map $g \colon G \to {}^* (H/\Gamma)$ such that
\begin{equation}\label{fF-nil}
 \int_G f \overline{F(\st g)}\ d\mu_G \neq 0.
\end{equation}
\end{conjecture}

We claim that Conjecture \ref{uk-inverse-nil} follows from Conjecture \ref{uk-inverse-nil-nonst} for any given choice of $s$.  It suffices to establish this theorem in the standard universe. Suppose for contradiction that Conjecture \ref{uk-inverse-nil} failed in that universe for some standard $\eta,s$. As is well known (using for instance the theory of Mal'cev bases \cite{malcev}), there are only countably many isomorphism classes of (standard) degree $s$ filtered nilmanifolds.  If we let $H_\n/\Lambda_\n$ be an enumeration of (standard) representatives of these classes, then from the axiom of choice we see that we may obtain a sequence 
$G_\n$ of standard finite additive groups and standard $1$-bounded functions $f_\n \colon G_\n \to \C$ with the property that $\|f_\n\|_{U^{s+1}(G_\n)} \geq \eta$, and furthermore that there does \emph{not} exist a degree $s$ filtered nilmanifold $H/\Gamma = H_{\n'}/\Lambda_{\n'}$ for some $\n' \leq n$, some standard Lipschitz function $F_\n \colon H/\Gamma \to \C$ with norm at most $\n$, and some standard polynomial map $g_\n \colon G_\n \to H/\Gamma$ such that
$$ |\E_{x \in G_\n} f_\n(x) \overline{F_\n(g_\n(x))}| \geq \frac{1}{\n}.$$
Now let $G \coloneqq \prod_{\n \to p} G_\n$ be the ultraproduct of the $G_\n$, and $f \coloneqq \st \lim_{\n \to p} f_\n$ be the standard part of the ultralimit of the $f_\n$.  Then we can endow $G$ with the Loeb probability space structure, and $f \in L^\infty(G)$.  From \eqref{flim} we have
$$ \|f\|_{U^{s+1}(G)} \geq \eta > 0.$$
We may then apply Conjecture \ref{uk-inverse-nil-nonst} to conclude that there exists a standard degree $s$ filtered nilmanifold $H/\Gamma = H_{\n_0}/\Lambda_{\n_0}$, a standard continuous function $F \colon H/\Gamma \to \C$, and an internal polynomial map $g \colon G \to {}^* (H/\Gamma)$ such that \eqref{fF-nil} holds.  By the Stone--Weierstrass theorem, $F$ is the uniform limit of Lipschitz functions, so we may assume without loss of generality that $F$ is Lipschitz.  Observe that $F(\st g) = \st F(g)$. Writing $g = \lim_{\n \to p} g_\n$ as the ultralimit of standard polynomial maps $g_\n \colon G_\n \to H/\Gamma$, we conclude that
$$ f \overline{F(\st g)} = \st \lim_{\n \to p} f_\n \overline{F(g_\n)}$$
and hence by \eqref{loeb}
$$ \st \lim_{\n \to p} \E_{x \in G_\n} f_\n(x) \overline{F(g_\n(x))} \neq 0.$$
In particular, there is a $p$-large set of $\n$ such that
$$ |\E_{x \in G_\n} f_\n(x) \overline{F(g_\n(x))}| \geq \frac{1}{\n}.$$
But this contradicts the construction of the $f_\n$ for $\n$ sufficiently large, giving the claim. 
 
\begin{remark} It is not difficult to also conversely establish that Conjecture \ref{uk-inverse-nil-nonst} follows from Conjecture \ref{uk-inverse-nil} for a given choice of $s$; we leave the details to the interested reader.
\end{remark}

To establish Theorem \ref{implication}, it thus suffices to show that Theorem \ref{hk-inverse} implies
the $s=2$ case of Conjecture \ref{uk-inverse-nil-nonst}.  This is the purpose of the next three sections of the paper.

\section{Correspondence principle}

In this section we create a separable $\Z^\omega$-system associated to the data of Conjecture \ref{uk-inverse-nil-nonst}, whose dynamics are related to the additive combinatorial structure of the internal function $f$ appearing in that conjecture, in the spirit of the Furstenberg correspondence principle \cite{furstenberg1977ergodic}.  To create this correspondence we follow \cite{tz-finite} and rely on random sampling methods, though we will use the language of nonstandard analysis in order to simplify some of the ``epsilon management'' present in the arguments of \cite{tz-finite}.

More precisely, we will establish the following claim.

\begin{proposition}[Correspondence principle]\label{corr}  Let $G$ be a nonstandard finite additive group, and let ${\mathcal F}$ be an at most countable subset of $L^\infty(G)$.  Then there exists an ergodic separable $\Z^\omega$-system $(X,\X, \mu_X, T)$, where $(X,\X,\mu_X)$ is a factor of $(G, {\mathcal L}_G, \mu_G)$ (thus $X=G$ as a set, $\X$ is a subalgebra of ${\mathcal L}_G$, and $\mu_X$ is the restriction of $\mu_G$ to $\X$, and in particular $L^\infty(X) \subset L^\infty(G)$) with the following properties:
\begin{itemize}
\item[(i)] (Equivalence of Host--Kra and Gowers inner products)   The Gowers inner products on $L^\infty(G)$ restrict to the Host--Kra inner products on $L^\infty(X)$.  In other words, for any $d \geq 0$ and any tuple $(f_\omega)_{\omega \in \{0,1\}^d}$ of functions $f_\omega \in L^\infty(X) \subset L^\infty(G)$ one has the identity
\begin{equation}\label{tfo}
\langle (f_\omega)_{\omega \in \{0,1\}^d} \rangle_{U^d(X)} =  \langle (f_\omega)_{\omega \in \{0,1\}^d} \rangle_{U^d(G)}.
\end{equation}
\item[(ii)]  (${\mathcal F}$ is modeled)  We have ${\mathcal F} \subset L^\infty(X)$.  In particular, from \eqref{tfo} one has $\|f\|_{U^d(X)} = \|f\|_{U^d(G)}$ for all $d \geq 1$ and $f \in {\mathcal F}$.
\end{itemize}
\end{proposition}

\begin{remark} The arguments in \cite[\S 3,4]{tz-finite}, when translated to the nonstandard setting, provide an analogue of this proposition in which $G$ is now a internally finite dimensional vector space over a standard finite field ${\mathbb F}_p$ of prime order, and the group $\Z^\omega$ is replaced by ${\mathbb F}_p^\omega$.  The proposition may also be compared with \cite[Proposition 3.12]{candela-szegedy-inverse}, in which $G$ is replaced by an ultraproduct of CFR coset nilspaces and $X$ is similarly replaced by a nilspace (with no dynamics).  A closely related correspondence principle was also established by Towsner \cite{towsner}. 
\end{remark}

We now prove this proposition.  By splitting the elements of ${\mathcal F}$ into real and imaginary parts we may assume without loss of generality that all the elements of ${\mathcal F}$ are real-valued.  

The idea is as follows.  The shift maps $T^h \colon G \to G$ will be defined on the entire group $G$ by the formula
\begin{equation}\label{action-form}
 T^h(x) \coloneqq x + \sum_{j=1}^\infty h_j \mathbf{g}_j
 \end{equation}
whenever $h = \sum_{j=1}^\infty h_j e_j$ and $x \in G$, for some suitable shifts $\mathbf{g}_j \in G$; note that these sums contain only finitely many non-zero terms.  In order to be able to verify the identity \eqref{tfo}, as well as the ergodicity of the system $(X,\X, \mu_X, T)$, it will be necessary that these shifts $\mathbf{g}_j$ are ``generic'' in a suitable sense relative to the function $f$ and its shifts.  To locate such generic shifts we shall use a probabilistic method, generating the shifts $\mathbf{g}_j$ randomly (i.e., as measurable functions of an (external) probability space $\Omega$) and verifying that these shifts obey the required properties almost surely.   This is analogous to how one can generate a uniquely ergodic system on the circle $\R/\Z$ by choosing a shift $\alpha$ uniformly at random from $\R/\Z$ and then considering the shift map $x \mapsto x+\alpha$, noting that the unique ergodicity property will hold for the resulting system almost surely (since $\alpha$ will almost surely be irrational).

We begin by setting up the probability space $\Omega$.  Write $G = \prod_{\n \to p} G_\n$ as the ultraproduct of finite additive groups $G_\n$.  By the Carath\'eodory extension theorem, for each $\n$ the (standard) product space $G_\n^\N$ of sequences $(g_{j,\n})_{j \in \N}$ can be given the structure of a (standard) probability space (the product of the probability space structures on $G_\n$ arising from the discrete $\sigma$-algebra and uniform probability measure).  One can view these standard probability spaces as modeling a standard sequence of independent uniformly distributed
random elements $g_{j,\n}$ of $G_\n$.  We then let $\Omega \coloneqq \prod_{\n \to p} (G_\n^\N)$ be the ultraproduct of these spaces; by the Loeb measure construction (see Appendix \ref{nonst}), we can view $\Omega$ as an (external) probability space.  We caution that we have the inequality
$$ \Omega = \prod_{\n \to p} (G_\n^\N) \neq \left(\prod_{\n \to p} G_\n\right)^\N = G^\N $$
in general; however there is certainly a measure-preserving map $\pi \colon \Omega \to G^\N$ (giving $G^\N$ the external probability space structure coming from taking the product of the probability spaces $G$) defined by
$$ \pi\left( \lim_{\n \to p} (g_{j,\n})_{j \in \N} \right) \coloneqq \left(\lim_{\n \to p} g_{j,\n}\right)_{j \in \N}$$
that one can easily verify to be a well-defined measure-preserving map by chasing all the definitions.  If one writes $\pi = (\mathbf{g}_j)_{j \in \N}$, one can then view the $\mathbf{g}_j \colon \Omega \to G$ as independent random variables taking values in $G$.  By construction, we see that if $J$ is a standard natural number and $F_\n \colon G_\n^J \to \C$ are a uniformly bounded sequence of standard functions, with limit $F \colon G^J \to \C$ defined by $F \coloneqq \st \lim_{\n \to p} F_n$, then
\begin{equation}\label{fident}
 \E F( \mathbf{g}_1,\dots,\mathbf{g}_J ) = \st \lim_{\n \to p} \E_{g_{1,\n},\dots,g_{J,\n} \in G_\n} F_\n(g_{1,\n},\dots,g_{J,\n}) 
 \end{equation}
 where on the left-hand side the symbol $\E$ denotes expectation with respect to the probability space $\Omega$.

\begin{remark} We caution that the limit function $F \colon G^J \to \C$ in \eqref{fident} is \emph{not} necessarily measurable with respect to the product $\sigma$-algebra ${\mathcal L}_G^J$ of the Loeb $\sigma$-algebra ${\mathcal L}_G$ of $G$; instead, it is merely measurable with respect to the larger $\sigma$-algebra ${\mathcal L}_{G^J}$ formed by applying the Loeb $\sigma$-algebra construction directly to $G^J = \prod_{\n \to p} (G_n^J)$.  To put it another way, the Loeb construction does \emph{not} commute with Cartesian products; see also Appendix \ref{nonst} for further discussion.  However, this lack of measurability will not be a problem in practice since the identity \eqref{fident} holds in all cases of interest.  Roughly speaking, this identity tells us that the random variables $\mathbf{g}_j$ are in fact ``better'' than merely independent copies of a single random variable $\mathbf{g}$ on $G$, because they can measure events in a larger $\sigma$-algebra ${\mathcal L}_{G^J}$ of $G^J$ than just the product $\sigma$-algebra of $J$ copies of ${\mathcal L}_G$.
\end{remark}

We now use these random variables $\mathbf{g}_j$ to define a (random) $\Z^\omega$-system $(X,\mu,T) = (X,\X, \mu_X, T)$ as follows.  We define a shift $T^h \colon G \to G$ for each $h \in \Z^\omega$ by the formula \eqref{action-form}
$$ T^h(x) \coloneqq x + \sum_{j=1}^\infty h_j \mathbf{g}_j$$
whenever $h = \sum_{j=1}^\infty h_j e_j$ and $x \in G$; note that these sums contain only finitely many non-zero terms.  That is to say, each generator $T^{e_j}$ is a shift in the random direction $\mathbf{g}_j$.  As usual we then define $T^h f \coloneqq f \circ T^{-h}$ for any $f \in L^\infty(G)$. It is then easy to verify that $(G, {\mathcal L}_G, \mu_G, T)$ is a $\Z^\omega$-system, but it will not be separable in general.  We therefore restrict attention to the separable factor $(X,\X,\mu_X,T)$ of $(G, {\mathcal L}_G, \mu_G,T)$ generated by ${\mathcal F}$; thus $\X$ is the (complete) $\sigma$-algebra generated by the $T^h f$ for $f \in {\mathcal F}$ and $h \in \Z^\omega$, and $\mu_X$ is the restriction of $\mu_G$ to $\X$.

From construction it is clear that $(X,\X,\mu_X,T)$ is a separable $\Z^\omega$-system (since there are only countably many functions of the form $T^h f$ up to $\mu_X$-almost everywhere equivalence), and
property (ii) of Theorem \ref{corr} is obeyed. It remains to (almost surely) establish ergodicity, as well as property (i).

We begin with (i).  Fix $d \geq 0$.  By homogeneity we may restrict the $f_\omega$ to be $1$-bounded. From H\"older's inequality one easily checks that
$$ \langle (f_\omega)_{\omega \in \{0,1\}^d} \rangle_{U^d(X)} \leq \prod_{\omega \in \{0,1\}^d} \|f_\omega\|_{L^{2^d}(X)}$$
for all $f_\omega \in L^\infty(X)$ (cf., \cite[\S 5.1]{tz-concat}).  Thus, the left-hand side of \eqref{tfo} is jointly continuous in the $f_\omega$ in the $L^{2^d}(X)$ topology.  A similar argument shows that the right-hand side of \eqref{tfo} is also continuous in this topology.  Thus we may restrict the $f_\omega$ to any $L^{2^d}(X)$-dense subset of the unit ball of $L^\infty(X)$.  By the Stone--Weierstrass theorem, an example of such a set is given by polynomial combinations of finitely many of the shifts $T^h f$ with rational coefficients with $f \in {\mathcal F}$.  This set is countable (modulo $\mu_X$-almost everywhere equivalence), so to show that (i) holds almost surely, it suffices to verify (i) for fixed choices of $f_\omega$, each of which is a polynomial combination of finitely many of the $T^h f$ with rational coefficients.  In particular, each $f_\omega$ takes the form
$$ f_\omega(x) = \st F_\omega( x, \mathbf{g}_1,\dots,\mathbf{g}_J )$$
for some standard $J$ and some (deterministic) internal $1$-bounded function $F_\omega \colon G \times G^J \to {}^* \C$.  It thus suffices to establish the identity
\begin{equation}\label{fang}
\langle (\st F_\omega(\cdot,\mathbf{g}_1,\dots,\mathbf{g}_J) )_{\omega \in \{0,1\}^d} \rangle_{U^d(X)} =  \langle (\st F_\omega(\cdot,\mathbf{g}_1,\dots,\mathbf{g}_J) )_{\omega \in \{0,1\}^d} \rangle_{U^d(G)}
\end{equation}
almost surely for any standard $J$ and any deterministic internal $1$-bounded functions $F_\omega \colon G \times G^J \to {}^* \C$ for $\omega \in \{0,1\}^d$.

We establish this by induction on $d$.  For $d=0$ both sides are equal to
$$ \st \E_{x \in G} F_{()}(x, \mathbf{g}_1, \dots, \mathbf{g}_J )$$
so the claim is immediate in this case.  Now suppose that $d \geq 1$ and that the claim has already been proven for $d-1$.  From \eqref{bo-ident}, the left-hand side of \eqref{fang} can be written as
$$ \lim_{n \to \infty} \E_{h^0, h^1 \in \Phi_n}
\langle (\st F_{\omega,h^0 \cdot \mathbf{g},h^1 \cdot \mathbf{g},\mathbf{g}_1,\dots,\mathbf{g}_J} )_{\omega \in \{0,1\}^{d-1}} \rangle_{U^{d-1}(X)}
$$
where
\begin{equation}\label{faag}
F_{\omega,a^0,a^1,g_1,\dots,g_J}(x) \coloneqq F_{\omega,0}(x - a^0, g_1,\dots,g_J)
\overline{F_{\omega,1}}(x - a^1, g_1,\dots,g_J),
\end{equation}
for $\omega \in \{0,1\}^{d-1}$ with $h^j = \sum_{i=1}^\infty h^j_i e_i$, and
$$ h^j \cdot \mathbf{g} \coloneqq \sum_{i=1}^\infty h^j_i \mathbf{g}_i$$
for $j=0,1$.  Note that each $F_{\omega,h^0 \cdot \mathbf{g},h^1 \cdot \mathbf{g},\mathbf{g}_1,\dots,\mathbf{g}_J}(x)$ is a $1$-bounded internal function of $x$ and finitely many of the $\mathbf{g}_i$. Applying the induction hypothesis, we may write the above expression as
$$ \lim_{n \to \infty} \E_{h^0, h^1 \in \Phi_n}
\langle (\st F_{\omega,h^0 \cdot \mathbf{g},h^1 \cdot \mathbf{g},\mathbf{g}_1,\dots,\mathbf{g}_J} )_{\omega \in \{0,1\}^{d-1}} \rangle_{U^{d-1}(G)}.
$$
We can rewrite this as
$$ \lim_{n \to \infty} \st \E_{h^0, h^1 \in \Phi_n} H(h^0 \cdot \mathbf{g}, h^1 \cdot \mathbf{g}, \mathbf{g}_1,\dots,\mathbf{g}_J) $$
where $H$ is the internal Gowers inner product
\begin{equation}\label{haag}
H(a^0,a^1,g_1,\dots,g_J) \coloneqq 
\langle (F_{\omega,a^0,a^1,g_1,\dots,g_J} )_{\omega \in \{0,1\}^{d-1}} \rangle_{{}^* U^{d-1}(G)}.
\end{equation}
Note that $H: G^2 \times G^J \to {}^* \C$ is an internal $1$-bounded function.  We will shortly show that one almost surely has the sampling identity
\begin{equation}\label{sampling}
\lim_{n \to \infty} \st \E_{h^0, h^1 \in \Phi_n} H(h^0 \cdot \mathbf{g}, h^1 \cdot \mathbf{g}, \mathbf{g}_1,\dots,\mathbf{g}_J) =
\st {}^* \E_{a^0, a^1 \in G} H(a^0, a^1, \mathbf{g}_1,\dots,\mathbf{g}_J)
\end{equation}
for any given $1$-bounded internal function $H: G^2 \times G^J \to {}^* \C$, where we recall that ${}^* \E_{a^0,a^1 \in G}$ denotes the internal average over $G \times G$.  Assuming this identity holds for the moment, we can now write the left-hand side of \eqref{fang} as
$$
\st {}^* \E_{a^0, a^1 \in G} H(a^0, a^1, \mathbf{g}_1,\dots,\mathbf{g}_J).$$
But from \eqref{haag}, \eqref{faag}, \eqref{inner-recurse} one has the internal identity
$$
{}^* \E_{a^0, a^1 \in G} H(a^0, a^1, g_1,\dots,g_J) = 
\langle F_\omega(\cdot, g_1,\dots,g_J) \rangle_{{}^* U^d(G)}$$
for any $g_1,\dots,g_J \in G$, as can be seen by first establishing the corresponding identity on each $G_\n$ and then taking ultralimits.  The claim \eqref{fang} now follows from \eqref{flim}.  

It remains to establish \eqref{sampling} almost surely for a given internal function $H$.  If $H(a^0,a^1,g_1,\dots,g_J)$ does not actually depend on the first two inputs $a^0,a^1$ then the claim is trivial.  Subtracting off the internal marginal mean ${}^* \E_{a^0, a^1 \in G} H(a^0, a^1, g_1,\dots,g_J)$ and using linearity, it thus suffices to verify \eqref{sampling} under the additional assumption of vanishing marginal mean
\begin{equation}\label{mean-zero}
{}^* \E_{a^0, a^1 \in G} H(a^0, a^1, g_1,\dots,g_J) = 0
\end{equation}
for all $g_1,\dots,g_J$, in which case our task is to show that
$$
\lim_{n \to \infty} \st \E_{h^0, h^1 \in \Phi_n} H(h^0 \cdot \mathbf{g}, h^1 \cdot \mathbf{g}, \mathbf{g}_1,\dots,\mathbf{g}_J) = 0$$
almost surely.  By the Borel--Cantelli lemma, it suffices to show that
$$ \sum_{n=1}^\infty  \E \st |\E_{h^0, h^1 \in \Phi_n} H(h^0 \cdot \mathbf{g}, h^1 \cdot \mathbf{g}, \mathbf{g}_1,\dots,\mathbf{g}_J)|^2 < \infty.$$
The left-hand side can be expanded as
\begin{equation}\label{sumo}
\sum_{n=1}^\infty  \E_{h^0, h^1,h^2,h^3 \in \Phi_n} \E \st H(h^0 \cdot \mathbf{g}, h^1 \cdot \mathbf{g}, \mathbf{g}_1,\dots,\mathbf{g}_J) \overline{H}(h^2 \cdot \mathbf{g}, h^3 \cdot \mathbf{g}, \mathbf{g}_1,\dots,\mathbf{g}_J)
\end{equation}
Now we take advantage of the specific form \eqref{folner} of the F{\o}lner sequence $\Phi_n$, which gives $h^i = \sum_{j=1}^{2^n} h^i_j e_j$ for some $h^i_j \in \{0,\dots,n\}$ for each $i=0,1,2,3$.  Suppose that there exists an odd number $J < j < 2^n$ such that $h^i_{j+i'} = 1_{i=i'}$ for all $i=0,1,2,3$ and $i'=0,1$.  Then for each $i'=0,1$, the random variable $\mathbf{g}_{j+i'}$ only makes an appearance in the expression
$$ \E \st H(h^0 \cdot \mathbf{g}, h^1 \cdot \mathbf{g}, \mathbf{g}_1,\dots,\mathbf{g}_J) \overline{H}(h^2 \cdot \mathbf{g}, h^3 \cdot \mathbf{g}, \mathbf{g}_1,\dots,\mathbf{g}_J)$$
only through the $h^i \cdot \mathbf{g}$ term, which can be written as the sum of $\mathbf{g}_{j+i'}$ plus another expression not depending on either of the $\mathbf{g}_{j}, \mathbf{g}_{j+1}$.  From \eqref{mean-zero} and a change of variable, we then have
$$ {}^* \E_{g \in G^{2^n}} H(h^0 \cdot g, h^1 \cdot g, g_1,\dots,g_J) \overline{H}(h^2 \cdot g, h^3 \cdot g, g_1,\dots,g_J) = 0$$
(writing $g = (g_1,\dots,g_{2^n})$ and defining the inner products $h^i \cdot g$ in the obvious fashion), and hence by \eqref{fident} 
$$ \E \st H(h^0 \cdot \mathbf{g}, h^1 \cdot \mathbf{g}, \mathbf{g}_1,\dots,\mathbf{g}_J) \overline{H}(h^2 \cdot \mathbf{g}, h^3 \cdot \mathbf{g}, \mathbf{g}_1,\dots,\mathbf{g}_J) = 0.$$
Thus in the expression \eqref{sumo} we may restrict attention to those $h^0,h^1,h^2,h^3$ for which there do \emph{not} exist any odd $J < j < 2^n$ such that $h^i_{j+i'} = 1_{i=i'}$ for all $i=0,1,2,3$ and $i'=0,1$.  For each given $j$, the proportion of $h^0,h^1,h^2,h^3$ for which this property fails is $1 - \frac{1}{(n+1)^8}$, and the events are independent for the $\lfloor \max(\frac{2^n-J}{2},0) \rfloor$ different choices of $j$, and hence we can upper bound \eqref{sumo} by
$$ \sum_{n=1}^\infty \left(1 - \frac{1}{(n+1)^8}\right)^{\lfloor \max(\frac{2^n-J}{2},0) \rfloor}.$$
This sum is easily seen to be finite (bounding $1-\frac{1}{(n+1)^8}$ by $\exp(-\frac{1}{(n+1)^8})$, and the claim follows. This concludes the proof of Proposition \ref{corr}(i).

It remains to establish that $X$ is ergodic.  If this is not the case, then there exists a non-trivial $\Z^\omega$-invariant function $F \in L^\infty(X)$ of mean zero, which we can take to be $1$-bounded.  By the properties of Loeb measure, one can write $F = \st \tilde F$ for some internal $1$-bounded function $\tilde F$.  From the $d=1$ case of Proposition \ref{corr}(i) and \eqref{flim}, \eqref{u1-form}, \eqref{loeb-int} we have
\begin{align*}
\langle F, F \rangle_{U^1(X)} &= \langle F, F \rangle_{U^1(G)} \\
&= \st \langle \tilde F, \tilde F \rangle_{{}^* U^1(G)} \\
&= \st |{}^* \E_{x \in G} \tilde F(x)|^2 \\
&= \left|\int_X F\ d\mu\right|^2 \\
&= 0.
\end{align*}
On the other hand, from \eqref{bo}, \eqref{bo-global} and the hypothesis that $F$ is invariant, we easily compute that
$$\langle F, F \rangle_{U^1(X)}  = \int_X |F|^2\ d\mu.$$
Since $F$ is non-trivial, we obtain the required contradiction.  This completes the proof of Proposition \ref{corr}.

\section{The structure of nilpotent quotient spaces}\label{nilquot}

Theorem \ref{hk-inverse} produces a quotient space $H/\Gamma$ where $H$ is a filtered locally compact group of degree $s=2$, and $\Gamma$ is a filtered lattice of $H$.  This quotient space is then compact, and the Host--Kra spaces $\HK^k(H/\Gamma) \subset (H/\Gamma)^{\{0,1\}^k}$ are similarly compact (this follows readily from the proof of \cite[Lemma E.10]{gt-linear}).  In the language of nilspaces, $H/\Gamma$ in fact has the structure of a compact nilspace; see \cite[Appendix A]{gmv} for further discussion.  

In this section we explore the nilspace structure of $H/\Gamma$ further (``forgetting'' for now about the dynamical structure of the shift $T$), in particular establishing that $H/\Gamma$ is an inverse limit (in the category of compact nilspaces) of nilmanifolds.  Our arguments will not be restricted to the $s=2$ case.  If the $H_i$ were connected then one could argue using tools such as \cite[Theorem 6.2]{candela-szegedy-inverse} to establish this claim, but in our setting we do not have any connectedness properties on $H$ and we will need to rely instead on the general structural theory of nilpotent locally compact groups that are not necessarily connected.

We first need a simple algebraic lemma concerning the ability to ``factor'' Host--Kra groups.  Given two subsets $A,B$ of a multiplicative group $G$, we define the product set $A B \coloneqq \{ a b: a \in A, b \in B \}$.  (In particular, the product $HK$ of two subgroups $H,K$ of $G$ need not be a subgroup if neither of $H,K$ are normal.)

\begin{lemma}[Factoring Host--Kra groups]\label{ghk-split}  Let $G = (G,\cdot)$ be a prefiltered group, and let $H, K$ be prefiltered subgroups of $G$ such that $G_i \subset H_i K_i$ for all $i \geq 0$.  Then $\HK^k(G) \subset \HK^k(H) \HK^k(K)$ for all $k$.
\end{lemma}

\begin{proof}  For any $0 \leq i \leq k+1$, let $S_i$ denote the set of all elements $(g_\omega)_{\omega \in \{0,1\}^k}$ of $\HK^k(G)$ such that $g_\omega = 1$ whenever $|\omega| < i$.  We claim the inclusion
\begin{equation}\label{si}
 S_i \subset \HK^k(H) S_{i+1} \HK^k(K)
\end{equation}
for all $0 \leq i \leq k$.  Iterating this and noting that $S_0 = \HK^k(G)$ and $S_{k+1} = \{1\}$, we obtain the claim.

It remains to establish \eqref{si}.  Let $(g_\omega)_{\omega \in \{0,1\}^k}$ be an element of $S_i$.  We order the elements of $\{0,1\}^k$ as $\omega_1,\dots,\omega_{2^k}$ in such a way that $|\omega_1|,\dots,|\omega_{2^k}|$ is non-decreasing.  Applying \cite[Proposition A.5]{gmv} (or \cite[(E.1)]{gt-linear}), we may factor
$$ (g_\omega)_{\omega \in \{0,1\}^k} = \prod_{j=1}^{2^k} [x_{\omega_j}]_{\omega_j}$$
as an ordered product for some $x_{\omega_j} \in G_{|\omega_j|}$, where the notation $[x_{\omega_j}]_{\omega_j}$ was defined in \eqref{go}.  Since this product lies in $S_i$, we have $x_{\omega_j} = 1$ whenever $|\omega_j| < i$.  When $|\omega_j|=i$, we may use the hypothesis $G_i \subset H_i K_i$ to factor $x_{\omega_j} = y_{\omega_j} z_{\omega_j}$ where $y_{\omega_j} \in H_i$ and $z_{\omega_j} \in K_i$.  We can then factor
$$ (g_\omega)_{\omega \in \{0,1\}^k} = \left(\prod_{1 \leq j \leq 2^k: |\omega_j|=i} [y_{\omega_j}]_{\omega_j}\right)
(g'_\omega)_{\omega \in \{0,1\}^k} \left(\prod_{1 \leq j \leq 2^k: |\omega_j|=i} [z_{\omega_j}]_{\omega_j}\right)$$
for some $(g'_\omega)_{\omega \in \{0,1\}^k}$ that one readily verifies to lie in $S_{i+1}$.  The claim follows.
\end{proof}

Using this lemma we can replace the filtered group $H$ by a compactly generated filtered group $H'$ without affecting the Host--Kra structure.

\begin{corollary}[Replacing with a compactly generated group]\label{compact}  Let $H$ be a filtered locally compact group of degree $s$, and let $\Gamma$ be a filtered lattice in $H$.  Then there exists a compactly generated open subgroup $H'$ of $H$ (which has the structure of a locally compact filtered group with $H'_i \coloneqq H' \cap H_i$) such that $\Gamma \cap H'$ is a filtered lattice of $H'$ with $H/\Gamma$ and $H'/(H'\cap \Gamma)$ isomorphic as compact topological spaces, with the identification also being a nilspace isomorphism.
\end{corollary}

\begin{proof}  Since $H_i/\Gamma_i$ is compact and $\Gamma_i$ is discrete, there exists a non-empty precompact open subset $K_i$ of $H_i$ such that $H_i = K_i \Gamma_i$.  If we let $H'$ be the group generated by $\bigcup_{i=0}^s K_i$, then $H'$ is a compactly generated open subgroup of $H$ and $H_i = H'_i \Gamma_i$ for all $i$.  In particular the obvious identification between $H_i/\Gamma_i$ and $H'_i/(H'_i \cap \Gamma_i)$ is a homeomorphism, which implies that $H'_i \cap \Gamma_i$ is cocompact in $H'_i$.  Thus $H' \cap \Gamma$ is a lattice in $H'$.  From Lemma \ref{ghk-split} we have
$$ \HK^k(H) = \HK^k(H') \HK^k(\Gamma).$$
Now we apply the canonical projection $\pi \colon H^{\{0,1\}^k} \to (H/\Gamma)^{\{0,1\}^k} \equiv (H'/H' \cap \Gamma)^{\{0,1\}^k}$ to both sides of this identity.  The image of $\HK^k(H)$ under $\pi$ is $\HK^k(H/\Gamma)$ by definition, while the image of $\HK^k(H') \HK^k(\Gamma)$ is 
$$ \pi(\HK^k(H') \HK^k(\Gamma)) = \pi(\HK^k(H')) \equiv \HK^k(H'/(H' \cap \Gamma)),$$
noting (from the group isomorphism theorems) that the restriction of $\pi$ to $(H')^{\{0,1\}^k}$ agrees (after the obvious identifications) with the projection from $(H')^{\{0,1\}^k}$ to $(H'/H' \cap \Gamma)^{\{0,1\}^k}$.  Thus we have $\HK^k(H/\Gamma) \equiv \HK^k(H'/(H'\cap \Gamma))$ for all $k$, giving the required nilspace isomorphism.
\end{proof}

The advantage of passing to a compactly generated setting is provided by the following variant of the Gleason--Yamabe theorem.

\begin{proposition}[Gleason--Yamabe for compactly generated nilpotent groups]\label{gleason}  Let $H$ be a compactly generated locally compact nilpotent group.  Then $H$ is the inverse limit of Lie groups.  In other words, every neighborhood $U$ of the identity in $H$ contains a compact normal subgroup $N$ such that $H/N$ is isomorphic to a Lie group.
\end{proposition}

\begin{proof}  This follows from \cite[Theorem 9]{hlm}.  We remark that this theorem is simpler to prove than the general Gleason--Yamabe theorem \cite{gleason-struct}, \cite{yamabe}, which asserts that any locally compact group contains an open subgroup that is an inverse limit of Lie groups.  The point is that (as shown in \cite{hlm}) the passage to an open subgroup is unnecessary in the compactly generated nilpotent case. 
\end{proof}

\begin{remark} The hypothesis of compact generation in Proposition \ref{gleason} is necessary.  For instance, if $p$ is a prime, the Heisenberg group $\begin{pmatrix} 1 & \Q_p & \Q_p \\ 0 & 1 & \Q_p \\ 0 & 0 & 1 \end{pmatrix}$ over the $p$-adic field $\Q_p$ is locally compact nilpotent, but not the inverse limit of Lie groups, basically because all compact normal subgroups of this group lie in the center $\begin{pmatrix} 1 & 0 & \Q_p \\ 0 & 1 & 0 \\ 0 & 0 & 1 \end{pmatrix}$.  For similar reasons the nilpotency hypothesis is necessary, as can be seen by considering the compactly generated solvable group $\Z \ltimes \Q_p$, where the action of $\Z$ on $\Q_p$ is generated by multiplication by $p$.
\end{remark}

We can combine Corollary \ref{compact} and Proposition \ref{gleason} to approximate continuous functions on a general degree $s$ quotient $H/\Gamma$ by (pullbacks of) continuous functions of degree $s$ nilmanifolds $\tilde H/\tilde \Gamma$.

\begin{corollary}[Approximation by nilmanifolds]\label{nil-approx}  Let $H$ be a filtered locally compact group of degree $s$, and let $\Gamma$ be a filtered lattice in $H$.  Let $F \colon H/\Gamma \to \C$ be a continuous map, and let $\eps>0$.  Then there exists a filtered nilmanifold $\tilde H/\tilde \Gamma$ of degree $s$, a continuous function $\tilde F \colon \tilde H/\tilde \Gamma \to \C$, and a continuous nilspace morphism $\pi \colon H/\Gamma \to \tilde H/\tilde \Gamma$ such that
$$ |F(x) - \tilde F(\pi(x))| \leq \eps$$
for all $x \in H/\Gamma$.
\end{corollary}

Another way of expressing this corollary is that every quotient $H/\Gamma$ of a filtered locally compact group of degree $s$ by a filtered lattice is an inverse limit (in the category of compact nilspaces) of degree $s$ filtered nilmanifolds.

\begin{proof}  We may assume without loss of generality that $F$ is real-valued. By Corollary \ref{compact} we may assume without loss of generality that $H$ is compactly generated, and is thus an inverse limit of Lie groups by Proposition \ref{gleason}, which are nilpotent since $H$ is nilpotent.  Since $F$ is continuous on the compact space $H/\Gamma$, it is uniformly continuous, thus there exists an open neighborhood $U$ of the identity in $H$ such that $|F(hx)-F(x)| \leq \eps$ for all $h \in U$.  By Proposition \ref{gleason}, we can thus find a compact normal subgroup $N$ of $H$ in $U$ such that $H/N$ is a nilpotent Lie group. Since $N \subset U$, we have $|F(hx)-F(x)| \leq \eps$ for all $h \in N$.  The map\footnote{One could also average here using the Haar measure on $N$ if desired.} $x \mapsto \min_{h \in N} F(hx)$ is then a continuous $N$-invariant function that lies within $\eps$ of $F$ in the uniform topology.  This function descends to a continuous function on the quotient $(H/N) / (\Gamma N/N)$, which is a filtered nilmanifold by Proposition \ref{gleason}.  The claim follows.
\end{proof}

\section{Concluding the implication of the Gowers norm inverse theorem from the Host--Kra inverse theorem}\label{stab-sec}

We can now conclude the proof of Theorem \ref{implication}.  By the discussion in Section \ref{nonst-impl}, it suffices to show that Conjecture \ref{uk-inverse-nil-nonst} holds for $s=2$.  Let $G, f$ be as in that conjecture.  We invoke 
Proposition \ref{corr} (with ${\mathcal F} = \{f\}$)  to construct an ergodic $\Z^\omega$-system $(X, \mu_X, T)$ obeying the conclusions of that proposition.  In particular
$$ \|f\|_{U^{s+1}(X)} = \|f\|_{U^{s+1}(G)} > 0.$$
Applying Theorem \ref{hk-inverse}, we can find a filtered locally compact group $H$ of degree at most $s=2$, a filtered lattice $\Gamma$ of $H$, a translation action $T_{H/\Gamma} \colon \Z^\omega \to \Aut(H/\Gamma)$ defined by
$$ T_{H/\Gamma}^h(x) \coloneqq \phi(h) x$$
for some group homomorphism $\phi \colon \Z^\omega \to H$, a $\Z^\omega$-morphism $\Pi \colon X \to H/\Gamma$, and a continuous function $F \colon H/\Gamma \to \C$ such that
$$ \int_X f(x) \overline{F(\Pi(x))}\ d\mu_X \neq 0.$$
Applying Corollary \ref{nil-approx} for a sufficiently small $\eps>0$ and using the triangle inequality, we can thus find a filtered nilmanifold $\tilde H/\tilde \Gamma$ of degree $s$, a continuous function $\tilde F \colon \tilde H/\tilde \Gamma \to \C$, and a continuous nilspace morphism $\pi \colon H/\Gamma \to \tilde H/\tilde \Gamma$ such that
\begin{equation}\label{intx}
\int_X f(x) \overline{\tilde F(\pi \circ \Pi(x))}\ d\mu_X \neq 0.
\end{equation}
We now make a key definition:

\begin{definition}  A Loeb-measurable map $\tilde g \colon G \to \tilde H/\tilde \Gamma$ is \emph{almost polynomial} if, for every standard $k \geq 0$, one has
$$ (\tilde g(x_\omega))_{\omega \in \{0,1\}^k} \in \HK^k(\tilde H/\tilde \Gamma)$$
for $\mu_{\HK^k(G)}$-almost all $(x_\omega)_{\omega \in \{0,1\}^k} \in \HK^k(G)$, where $\mu_{\HK^k(G)}$ is Loeb measure on $\HK^k(G)$.
\end{definition}

\begin{lemma}  The map $\pi \circ \Pi \colon G \to \tilde H/\tilde \Gamma$ is almost polynomial.
\end{lemma}

\begin{proof}  Fix $k \geq 0$. By the second countable nature of $(\tilde H/\tilde \Gamma)^{\{0,1\}^k}$, it suffices to show that for any open sets $U_\omega \subset \tilde H/\tilde \Gamma$, $\omega \in \{0,1\}^k$ with $\prod_{\omega \in \{0,1\}^k} U_\omega$ disjoint from $\HK^k(\tilde H/\tilde \Gamma)$, that the set
$$ \{ (x_\omega)_{\omega \in \{0,1\}^k} \in \HK^k(G): \pi(\Pi(x_\omega)) \in U_\omega \forall \omega \in \{0,1\}^k \}$$
is $\mu_{\HK^k(G)}$-null, or equivalently that
\begin{equation}\label{varph}
 \int_{\HK^k(G)} \prod_{\omega \in \{0,1\}^k} 1_{\pi^{-1}(U_\omega)}(\Pi(x_\omega))\ d\mu_{\HK^k(G)}((x_\omega)_{\omega \in \{0,1\}^k}) = 0.
 \end{equation}
From \eqref{flim}, we can write the left-hand side of \eqref{varph} as
 $$ \langle (1_{\pi^{-1}(U_\omega)} \circ \Pi)_{\omega \in \{0,1\}^k} \rangle_{U^k(G)}$$
 which by \eqref{tfo} is equal to
 $$ \langle (1_{\pi^{-1}(U_\omega)} \circ \Pi)_{\omega \in \{0,1\}^k} \rangle_{U^k(X)}.$$
Since $\Pi \colon X \to H/\Gamma$ is a $\Z^\omega$-morphism, we may write this as
 $$ \langle (1_{\pi^{-1}(U_\omega)})_{\omega \in \{0,1\}^k} \rangle_{U^k(H/\Gamma)}.$$
By \eqref{bo}, \eqref{bo-global}, it thus suffices to show that
$$ \int_{H/\Gamma} \prod_{\omega \in \{0,1\}^k} T_{H/\Gamma}^{h^{\omega_1}_1 + \dots + h^{\omega_k}_k} 1_{\pi^{-1}(U_\omega)} = 0$$
for any $h^0_1,h^1_1, \dots, h^0_k,h^1_k \in \Z^\omega$.  By \eqref{thh}, it thus suffices to show that
$$ \left( \phi( h^{\omega_1}_1 + \dots + h^{\omega_k}_k ) x \right)_{\omega \in \{0,1\}^k} \not \in \prod_{\omega \in \{0,1\}^k} \pi^{-1}(U_\omega).$$
But the left-hand side lies in $\HK^k(H/\Gamma)$, which is in $(\pi^{\{0,1\}^k})^{-1} (\tilde H/\tilde \Gamma)$, and the claim then follows since  $\prod_{\omega \in \{0,1\}^k} U_\omega$ is disjoint from $\HK^k(\tilde H/\tilde \Gamma)$.
\end{proof}

We now invoke the following stability result:

\begin{lemma}[Stability of polynomials]\label{stab-poly}  If $\tilde g \colon G \to \tilde H/\tilde \Gamma$ is almost polynomial, then there exists an internal polynomial map $g \colon G \to {}^*( \tilde H/\tilde \Gamma )$ such that $\tilde g(x) = \st g(x)$ for $\mu_G$-almost all $x$.
\end{lemma}

This is a special case of \cite[Theorem 4.2]{candela-szegedy-inverse}, translated to the language of nonstandard analysis; the argument is valid for arbitrary degrees $s$.  For the convenience of the reader we supply a self-contained proof of this result in Appendix \ref{polystab}.  In the special case where $G$ is a vector space over a field ${\mathbb F}_p$ of prime order $p$, this result (phrased in the language of standard analysis) was established in \cite[Lemma 4.5]{tz-finite}.  We remark that the converse implication of Lemma \ref{stab-poly} (that functions that agree almost everywhere with the standard part of an internal polynomial are almost polynomial) is trivial.

Combining these two lemmas, we obtain an internal polynomial map $g \colon G \to {}^*( \tilde H/\tilde \Gamma )$ such that
$$ \pi \circ \Pi(x) = \st g(x)$$
for $\mu_G$-almost all $x$, and \eqref{fF-nil} now follows from \eqref{intx} (recalling that $\mu_X$ is the restriction of $\mu_G$ to $\X$).  This concludes the proof of Theorem \ref{implication}.

\appendix

\section{A review of nonstandard analysis}\label{nonst}

In this appendix we briefly review some basic definitions and tools from nonstandard analysis that we will use in this paper.  This formalism is frequently used in this subject (see e.g., \cite{gtz}, \cite{candela-szegedy-inverse}, \cite{tz-concat}); our formalism here is particularly close to that in \cite[\S 5]{tz-concat}.

We assume the existence of a \emph{standard universe} ${\mathfrak U}$ that contains all the objects implicitly discussed in results such as Theorem \ref{uk-inverse-nil}, which will henceforth be referred to as \emph{standard objects}.  For instance, the group $G$ in Theorem \ref{uk-inverse-nil} can be assumed to be a standard finite additive group, the nilmanifold $H/\Gamma$ can be assumed to be a standard filtered nilmanifold, and so forth.  We will shortly also introduce a nonstandard universe ${}^* {\mathfrak U}$ that contains the nonstandard objects, and refer to \emph{external objects} as objects that do not necessarily lie in either the standard or nonstandard universes.  (For instance, a standard or nonstandard object can either be viewed internally with respect to the universe it lies in, or externally without explicit mention of any universe.)

Throughout this paper we fix an (external) non-principal ultrafilter $p \in \beta \N \backslash \N$, which induces notions of ultraproduct, ultrapower, and ultralimit as defined for instance in \cite[\S 5]{tz-concat}.  Given some standard notion of a space (e.g., a finite abelian group, or a nilmanifold), a \emph{nonstandard space} will be an ultraproduct $X = \prod_{\n \to p} X_\n$ of standard spaces $X_\n$, whose elements are ultralimits $x = \lim_{\n \to p} x_\n$ of standard objects $x_\n \in X_\n$ for a $p$-large set of $\n$.  Thus for instance a nonstandard finite additive group is an ultraproduct $G = \prod_{\n \to p} G_\n$ of standard finite additive groups $G_\n$.  There is an obvious identification
$$ \left(\prod_{\n \to p} X_\n\right) \times  \left(\prod_{\n \to p} Y_\n\right) \equiv \prod_{\n \to p} (X_\n \times Y_\n)$$
for any nonstandard spaces $\prod_{\n \to p} X_\n$, $\prod_{\n \to p} Y_\n$; by abuse of notation, we shall make frequent use of this identification in this paper without further comment, so that the Cartesian product of finitely many nonstandard spaces is again a nonstandard space. (We caution however that such an identification breaks down for infinite Cartesian products.)

By {\L}os's theorem, every structural property of standard spaces that is expressible as a sentence in first-order logic is inherited by their nonstandard counterparts.  For instance, a nonstandard finite additive group will remain an additive group in the external sense, though it will usually not be externally finite.  The ultrapower $\prod_{\n \to p} X$ of a standard space $X$ will be denoted ${}^* X$, and contains $X$ as a subspace via the identification $x \equiv \lim_{\n \to p} x$.  For instance, ${}^* \R$ will denote the nonstandard real numbers (also known as the hyperreals), ${}^* \C$ will denote the nonstandard complex numbers, and all nonstandard objects lie in the nonstandard universe ${}^* {\mathfrak U}$.  The nonstandard reals ${}^* \R$ form an ordered field that contains the standard reals $\R$ as a subfield.

Let $X = \prod_{\n \to p} X_\n$ and $Y = \prod_{\n \to p} Y_\n$ be nonstandard sets.  An \emph{internal function} $f \colon X \to Y$ between these sets is a function that is an ultralimit $f = \lim_{\n \to p} f_\n$ of standard functions $f_\n \colon X_\n \to Y_\n$.  One can similarly define internal homomorphisms between nonstandard groups, or internal polynomial maps between nonstandard prefiltered groups, in the obvious fashion.  Given a standard function $f \colon X \to Y$ between standard sets, we define its nonstandard counterpart ${}^* f \colon {}^* X \to {}^* Y$ between the associated ultrapowers by ${}^* f \coloneqq \lim_{\n \to p} f$.  In some cases we may abuse notation by abbreviating ${}^* f$ as $f$.

A nonstandard complex number $z$ is said to be \emph{bounded} if $|z| \leq C$ for some standard real $C>0$, and \emph{infinitesimal} if $|z| \leq \eps$ for all standard reals $\eps>0$.  We denote infinitesimal complex numbers by $o(1)$.  By a well-known adaptation of the proof of the Bolzano--Weierstrass theorem, we see that to every bounded nonstandard complex number $z$ there is a unique standard complex number $\st(z)$ (known as the \emph{standard part} of $z$) such that $z = \st(z) + o(1)$.  The map $z \mapsto \st(z)$ is a $*$-homomorphism from the external algebra of bounded nonstandard complex numbers to $\C$.  In a similar spirit, given any compact metric space $X$, there is a standard part map $\st \colon {}^* X \to X$ which is the unique map such that ${}^* d(x, \st(x)) = o(1)$ for all $x \in {}^* X$, where ${}^* d \colon {}^* X \times {}^* X \to {}^* \R$ is the nonstandard counterpart of the metric $d \colon X \times X \to \R$.  In fact the standard part map is well defined for any compact \emph{metrizable} space, since one easily checks that any two metrics that generate the same topology also generate the same standard part map.

An important construction for us will be the \emph{Loeb measure construction} \cite{loeb}.  Given a sequence $(X_\n, \X_\n, \mu_\n)$ of standard probability spaces, this construction yields an (external) complete probability space $(X, \X, \mu_X)$ where $X = \prod_{\n \to p} X_\n$ is the ultraproduct of the $X_\n$, with the following two properties:
\begin{itemize}
\item[(i)] Whenever the internal function $f = \lim_{\n \to p} f_\n$ is the ultralimit of uniformly bounded standard measurable functions $f_\n \colon X_\n \to \C$, the function $\st f\colon X \to \C$ is bounded and measurable in $(X,\X,\mu)$, and
\begin{equation}\label{loeb}
\int_X \st f\ d\mu_X = \st \int_X f\ d{}^* \mu_X
\end{equation}
where the \emph{internal integral} $\int_X f\ d{}^* \mu_X$ of $f$ is defined by the formula
$$ \int_X f\ d{}^* \mu_X \coloneqq \lim_{\n \to p} \int_{X_\n} f_\n\ d\mu_\n.$$
\item[(ii)]  Conversely, whenever\footnote{Here we adopt the usual convention of identifying two elements on $L^\infty(X)$ whenever they agree $\mu_X$-almost everywhere.} $f \in L^\infty(X) = L^\infty(X,\mu_X)$, there exists a sequence $f_\n \colon X_\n \to \C$ of uniformly bounded standard measurable functions such that $f = \st \lim_{\n \to p} f_\n$ $\mu_X$-almost everywhere.  
\end{itemize}
See for instance \cite[\S 5]{tz-concat} for the construction of Loeb measure (based on the Carath\'eodory extension theorem).  One can then define the Lebesgue spaces $L^p(X,\mu_X)$ for $1 \leq p \leq \infty$ in the usual fashion.  We caution that these measure spaces will not be separable in general.

As a key special case of the Loeb measure construction, if $G$ is a nonstandard finite additive group, thus $G=\prod_{\n \to p} G_\n$ is the ultraproduct of standard finite additive groups $G_\n$, then one can give $G$ the structure of an (external) probability space $(G, {\mathcal L}_G, \mu_G)$, by endowing each of the standard finite groups $G_\n$ with the discrete $\sigma$-algebra and uniform measure, and applying the Loeb construction; we refer to $(G,{\mathcal L}_G,\mu_G)$ as the \emph{Loeb probability space} associated with $G$, with ${\mathcal L}_G$ the \emph{Loeb $\sigma$-algebra} and $\mu_G$ the \emph{Loeb measure} on $G$.  Thus for instance if the internal function $f = \lim_{\n \to p} f_\n$ is the ultralimit of uniformly bounded standard functions  $f_\n \colon G_\n \to \C$, then
\begin{equation}\label{loeb-int}
\int_G \st f\ d\mu_G = \st {}^* \E_{x \in G} f(x)
\end{equation}
where the \emph{internal average} ${}^* \E_{x \in G} f(x)$ of $f$ in $G$ is defined as
$$ {}^* \E_{x \in G} f(x) \coloneqq \lim_{\n \to p} \E_{x \in G_\n} f_\n(x).$$
Because counting measure on each $G_\n$ is translation-invariant, Loeb measure on $G$ is also translation-invariant.  Thus Loeb measure $\mu_G$ is very analogous to Haar measure, though we caution that Loeb measure is not actually a special case of Haar measure since $G$ does not have the structure of a locally compact group (and ${\mathcal L}_G$ is not a Borel $\sigma$-algebra).

If $G, H$ are nonstandard non-empty finite sets, then we can construct Loeb measure spaces $(G, {\mathcal L}_G, \mu_G)$, $(H, {\mathcal L}_H, \mu_H)$ on the nonstandard sets $G,H$, as well as a Loeb measure space $(G \times H, {\mathcal L}_{G \times H}, \mu_{G \times H})$ on the product space $G \times H$.  We caution that in general, the Loeb measure space is \emph{not} the product $(G \times H, {\mathcal L}_G \times {\mathcal L}_H, \mu_G \times \mu_H)$ of the individual Loeb measure spaces $(G, {\mathcal L}_G, \mu_G)$, $(H, {\mathcal L}_H, \mu_H)$, even after completing that product measure; see \cite[Remark 2.10.4]{tao-hilbert}.  Instead, the former space is an \emph{extension} of the latter: ${\mathcal L}_G \times {\mathcal L}_H$ is a subalgebra of ${\mathcal L}_{G \times H}$, and $\mu_G \times \mu_H$ is the restriction of $\mu_{G \times H}$ to ${\mathcal L}_G \times {\mathcal L}_H$.  Furthermore, we have the pleasant fact that the familiar Fubini theorem for $\mu_G \times \mu_H$ extends to $\mu_{G \times H}$:

\begin{lemma}[Keisler--Fubini theorem]\label{fubini}  Let $G, H$ be nonstandard non-empty finite sets.
\begin{itemize}
    \item[(i)] If $f \colon G \times H \to \C$ is a ${\mathcal L}_{G \times H}$-measurable bounded function, then
\begin{align*}
\int_{G \times H} f(g,h)\ d\mu_{G \times H}(g,h) &= \int_G \left(\int_H f(g,h)\ d\mu_H(h)\right)\ d\mu_G(g) \\
&= \int_H \left(\int_G f(g,h)\ d\mu_G(g)\right)\ d\mu_H(h).
\end{align*}
    \item[(ii)]  If $P(g,h)$ is an ${\mathcal L}_{G \times H}$-measurable property of elements $(g,h)$ of $G \times H$ (i.e., $\{ (g,h) \in G\times H: P(g,h) \hbox{ holds}\}$ is measurable in ${\mathcal L}_{G \times H}$, then $P(g,h)$ holds for $\mu_{G \times H}$-almost all $(g,h) \in G \times H$ if and only if, for $\mu_G$-almost all $g \in G$, $P(g,h)$ holds for $\mu_H$-almost all $h \in H$.
\end{itemize}
Similarly with the roles of $G$ and $H$ reversed.
\end{lemma}

In this paper we will only need part (ii) of this lemma; part (i) is only stated here in order to prove part (ii).

\begin{proof} For part (i), see for instance \cite[Theorem 2.10.3]{tao-hilbert}.  As with the usual Fubini theorem, one can weaken ``bounded'' here to ``absolutely integrable'' (and there is a Tonelli-type variant in which $f$ instead takes values in $[0,+\infty]$), but we will not need these extensions of this theorem here.  Part (ii) follows by applying part (i) to the indicator function of $P$.
\end{proof}

\section{Proof of stability of polynomials}\label{polystab}

In this appendix we establish Lemma \ref{stab-poly}.  We need some preparatory lemmas.  First, we show that an ``almost homomorphism'' can always be uniquely completed to an actual homomorphism.

\begin{lemma}[Repairing an almost homomorphism]\label{repair} Let $G = (G,\cdot), K = (K,\cdot)$ be nonstandard  groups with $G$ nonstandard finite, and let $\phi \colon G \to K$ be an internal map which is an ``almost homomorphism'' in the sense that $\phi(gh) = \phi(g) \phi(h)$ for $\mu_{G^2}$-almost all $(g,h) \in G^2$.  Then there exists a unique internal homomorphism $\tilde \phi \colon G \to K$ such that $\phi(g) = \tilde \phi(g)$ for $\mu_G$-almost all $g \in G$.  Finally, if $\phi$ takes values in an external subgroup $K'$ of $K$, then so does $\tilde \phi$.
\end{lemma}

\begin{proof}  Let $g \in G$.  We claim that for $\mu_{G^2}$-almost all $a,b \in G$, one has
$$ \phi(a)^{-1} \phi(ag)  = \phi(b)^{-1} \phi(bg).$$
Indeed, writing $b = ca$, we see from Lemma \ref{fubini}(ii) and translation invariance that for $\mu_{G^2}$-almost all $a,b \in G$, one has
$$ \phi(b) = \phi(c) \phi(a)$$
and
$$ \phi(bg) = \phi(c) \phi(ag)$$
giving the claim.  Thus, if we define $\tilde \phi(g)$ to be the mode of $\phi(a)^{-1} \phi(ag)$ for $a \in G$ (that is to say, the unique value that this quantity attains the most often), $\tilde \phi$ is an internal function with the property that for all $g \in G$, one has
\begin{equation}\label{ohi}
 \phi(a)^{-1} \phi(ag) = \tilde \phi(g)
 \end{equation}
for $\mu_G$-almost every $a \in G$.  In particular, for any $g,h \in G$, one has
\begin{align*}
\phi(a)^{-1} \phi(agh) &= \tilde \phi(gh) \\
\phi(a)^{-1} \phi(ag) &= \tilde \phi(g) \\
\phi(ag)^{-1} \phi(agh) &= \tilde \phi(h) 
\end{align*}
for at least one $a \in G$, hence $\tilde \phi(gh) = \tilde \phi(g) \tilde \phi(h)$.  Hence $\tilde \phi$ is a homomorphism.  Finally, from Lemma \ref{fubini}(ii) and \eqref{ohi} we have $\tilde \phi(g) = \phi(g)$ for $\mu_G$-almost all $g$.  This gives existence.  For uniqueness, observe that if we had two internal homomorphisms $\tilde \phi, \tilde \phi'$ obeying the conclusions of this lemma, then for any $g \in G$, we see from invariance of Loeb measure that
$$ \tilde \phi(g) = \tilde \phi(gh) \tilde \phi(h)^{-1} = \phi(gh) \phi(h)^{-1}$$
and
$$ \tilde \phi'(g) = \tilde \phi'(gh) \tilde \phi'(h)^{-1} = \phi(gh) \phi(h)^{-1}$$
for almost all $h \in G$, hence $\tilde \phi = \tilde \phi'$. This argument also shows that $\tilde \phi$ takes values in any external group $K'$ that $\phi$ does.
\end{proof}

Next, we establish some rigidity properties of near-polynomial maps.

\begin{lemma}[Infinitesimal near-polynomials are constant]\label{near-const}  Let $G$ be a nonstandard finite additive group, let $H$ be a standard prefiltered Lie group of degree $s$ for some $s \geq 0$, and $\phi \colon G \to {}^* H$ be an internal function such that for all standard $k \geq 0$, one has
$$ \aderiv_{h_1} \dots \aderiv_{h_k} \phi(x) \in {}^* H_k$$
for $\mu_{G^{k+1}}$-almost all $(x,h_1,\dots,h_k) \in G^{k+1}$, and such that $\st \phi(x) = 1$ for $\mu_G$-almost all $x \in G$.  Then there exists $c \in {}^* H$ with $\st c = 1$ such that $\phi(x)=c$ for $\mu_G$-almost all $x \in G$.
\end{lemma}

\begin{proof}  We may assume that the claim is already established for any smaller value of $s$.  By Lemma \ref{fubini}(ii), we see that if $s \geq 1$, then for $\mu_G$-almost all $h \in G$, the function $\aderiv_h \phi \colon G \to {}^* H$ obeys all the hypotheses of the lemma, but with $s$ replaced by $s-1$ and the prefiltration $(H_i)_{i=0}^\infty$ replaced by the shifted filtration $(H_{i+1})_{i=0}^\infty$.  By induction hypothesis, this implies that for such $h$ there exists $\eps_h \in {}^* H$ with $\st \eps_h = 1$ such that $\aderiv_h \phi(x) = \eps_h$ for $\mu_G$-almost all $x \in G$.  The same claim also holds in the $s=0$ case (with $\eps_h=1$ in this case).

Clearly for each $h$ the constant $\eps_h$, if it exists, is unique, and depends internally on $h$ (it is the mode of $\aderiv_h \phi$).  From the cocycle equation
$$\aderiv_{h+k} \phi(x) = \aderiv_h \phi(x) \aderiv_k\phi(x-h)$$
and Lemma \ref{fubini}(ii), we see that
$$ \eps_{h+k} = \eps_h \eps_k$$
for $\mu_{G^2}$-almost all $(h,k) \in G^2$.  By Lemma \ref{repair}, this means that there is an internal homomorphism $\tilde \eps \colon G \to {}^* H$ such that $\eps_h = \tilde \eps_h$ for $\mu_G$-almost all $(h,g) \in G^2$.  Thus $\st \tilde \eps_h = 1$ for $\mu_G$-almost all $h \in G$, and hence for all $h \in G$ by the homomorphism property; thus $\tilde \eps_h$ always stays infinitesimally close to the identity.  Since the Lie group $H$ contains no small subgroups (e.g., see \cite[Exercise 1.4.19]{tao-hilbert}), the nonstandard group ${}^* H$ contains no infinitesimal internal subgroups, thus the internal homomorphism $h \mapsto \tilde \eps_h$ must be trivial.  We conclude that $\aderiv_h \phi(x)=1$ for $\mu_G$-almost all $x \in G$, hence by Lemma \ref{fubini}(ii) and a change of variables we have $\phi(x) = \phi(y)$ for $\mu_{G^2}$-almost all $x,y \in G$.  By another application of Lemma \ref{fubini}(ii), we conclude that there exists $y \in G$ such that $\phi(x) = \phi(y)$ for $\mu_G$-almost all $y \in G$, and such that $\st \phi(y) = 1$, and the claim follows.
\end{proof}

\begin{corollary}[Nearby polynomial maps differ by a constant]\label{close}  Let $G$ be a nonstandard finite additive group, and $H/\Gamma$ a standard filtered nilmanifold.  If $\phi_1, \phi_2 \colon G \to {}^* (H/\Gamma)$ are internal polynomial maps such that $\st \phi_1(x) = \st \phi_2(x)$ for $\mu_G$-almost all $x \in G$, then there exists $c \in {}^* H$ with $\st c = 1$ such that $\phi_1 = c \phi_2$.  In particular, if $\phi_1,\phi_2$ agree at one point, then they are identical.
\end{corollary}

\begin{proof}  By hypothesis, we can find an internal map $\phi \colon G \to {}^* H$ such that $\phi_1(x) = \phi(x) \phi_2(x)$ and $\st \phi(x) = 1$ for $\mu_G$-almost all $x \in G$. We claim that for any standard $k \geq 0$, one has
$$ \aderiv_{h_1} \dots \aderiv_{h_k} \phi(x) \in {}^* H_k$$
for $\mu_{G^{k+1}}$-almost all $(x,h_1,\dots,h_k) \in G^{k+1}$.  This in turn will follow if we can show that the tuple
$$ \Phi_{x,h} \coloneqq (\phi(x + h \cdot \omega))_{\omega \in \{0,1\}^k} \in ({}^* H)^{\{0,1\}^k}$$
lies in $\HK^k( {}^* H )$
for $\mu_{G^{k+1}}$-almost all $(x,h_1,\dots,h_k) \in G^{k+1}$.  But by Lemma \ref{fubini}(ii), we see that for $\mu_{G^{k+1}}$-almost all $(x,h_1,\dots,h_k) \in G^{k+1}$, this tuple $\Phi_{x,h}$ is infinitesimally close to the identity and obeys the identity
\begin{equation}\label{phiphiphi}
 \Phi_{1,x,h} = \Phi_{x,h} \Phi_{2,x,h}
 \end{equation}
in $({}^* H)^{\{0,1\}^k} / ({}^* \Gamma)^{\{0,1\}^k}$, where for $i=1,2$ we define the points $\Phi_{i,x,h} \in ({}^* H)^{\{0,1\}^k} / ({}^* \Gamma)^{\{0,1\}^k}$ by the formula
$$ \Phi_{i,x,h} \coloneqq (\phi_1(x + h \cdot \omega))_{\omega \in \{0,1\}^k}.$$
Since $\phi_1,\phi_2$ are internal polynomials,In particular, $\Phi_{1,x,h}$ and $\Phi_{2,x,h}$ are infinitesimally close in $\HK^k({}^* H)/\HK({}^* \Gamma)$ (which we identify with a subset of $({}^* H)^{\{0,1\}^k} / ({}^* \Gamma)^{\{0,1\}^k}$ in the obvious fashion); since $\Gamma$ is cocompact in $H$, we can then write
$$ \Phi_{2,x,h} = \Psi_{x,h} ({}^* \Gamma)^{\{0,1\}^k}; \quad 
 \Phi_{1,x,h} = \eps_{x,h} \Psi_{x,h} ({}^* \Gamma)^{\{0,1\}^k}$$
for some $\eps_{x,h}, \Psi_{x,h} \in \HK^k({}^* H)$ with $\eps_{x,h}$ infinitesimally close to the identity and $\Psi_{x,h}$ bounded.  Inserting this back into \eqref{phiphiphi} and rearranging, we conclude that
$$ \Psi_{x,h}^{-1} \eps_{x,h}^{-1} \Phi_{x,h} \Psi_{x,h} \in ({}^* \Gamma)^{\{0,1\}^k}.$$
The left-hand side is infinitesimally close to the identity.  Since $\Gamma$ is discrete, the left-hand side must therefore equal the identity.  Thus $\Phi_{x,h} = \eps_{x,h}$, and thus $\Phi_{x,h}$ lies in $\HK^k( {}^* H )$ as claimed.

Applying Lemma \ref{near-const}, we conclude that there exists $c \in {}^* H$ with $\st c = 1$ such that $\phi_1(x) = c \phi_2(x)$ for $\mu_G$-almost all $x \in G$.  By Lemma \ref{fubini}(ii), this implies that for all $x \in G$, we have for $\mu_{G^{s+1}}$-almost all $h_1,\dots,h_{s+1} \in G$ that
$$ \phi_1(x + \omega \cdot h ) = c  \phi_2(x + \omega \cdot h ) $$
for all $\omega \in \{0,1\}^{s+1} \backslash \{0\}^{s+1}$.  Since $\phi_1, c\phi_2$ are both polynomial maps, one has
$$ \aderiv_{h_1} \dots \aderiv_{h_{s+1}} \phi_1(x) =  \aderiv_{h_1} \dots \aderiv_{h_{s+1}} c\phi_2(x) = 1$$
and hence $\phi_1(x)=\phi_2(x)$ for all $x \in G$, giving the claim.
\end{proof}

\begin{lemma}[Cocycle triviality on vector spaces]\label{cocy-triv}  Let $V$ be a standard finite-dimensional vector space, let $G$ be a nonstandard finite additive group, let $E$ be an internal subset of $G$ of full Loeb measure. Let
$$ c \colon \{ (h,k) \in E \times E: h+k \in E \} \to {}^* V$$
be an internal function obeying the cocycle equation
\begin{equation}\label{cocycl}
 c(h,k) + c(h+k,l) = c(h,k+l) + c(k,l)
 \end{equation}
whenever $h,k,l \in E$ are such that $h+k,k+l,h+k+l \in E$.  Then there exists an internal function $f \colon E \to {}^* V$ such that
\begin{equation}\label{cff}
 c(h,k) = f(h+k) -f(h) - f(k).
 \end{equation}
Furthermore, if we place a norm $\|\cdot\|_V$ on $V$ (and hence an internal norm $\|\cdot \|_{{}^* V}$ on ${}^* V$) we can choose $f$ so that $\|f\|_\infty \leq 2 \|c\|_\infty$, where
$$ \|f\|_\infty \coloneqq  \sup_{h \in E} \|f(h)\|_{{}^* V}$$
and
$$ \|c\|_\infty \coloneqq \sup_{h,k \in E: h+k \in E} \|c(h,k)\|_{{}^* V}$$
and we use the internal supremum.
\end{lemma}

\begin{proof}  Fix $V$, and let $\eps>0$ be a sufficiently small standard constant.  It suffices to prove the following standard version of the above statement: if $G$ is a standard finite additive group, $E$ is a subset of $G$ of density at least $1-\eps$, and $c \colon \{ (h,k) \in E \times E: h+k \in E \} \to V$ obeys \eqref{cocycl} whenever $h,k,l \in E$ are such that $h+k,k+l,h+k+l \in E$, then there exists $f \colon E \to V$ with $\|f\|_\infty \leq 2 \|c\|_\infty$ obeying \eqref{cff}.

Recall that in this paper we use the standard asymptotic notation of using $O(X)$ to denote a quantity bounded in magnitude by $CX$ for some absolute constant $C$. 
Define the averaged function $f \colon E \to V$ by the formula
$$ f(h) \coloneqq \frac{1}{|E|} \sum_{k \in E} c(h,k).$$
For any $h,k \in E$ with $h+k \in E$, we average \eqref{cocycl} over all $l \in E$ with $k+l, h+k+l \in E$, which is a subset of $E$ that omits only $O(\eps |E|)$ of the elements of $E$, to conclude that
$$ c(h,k) + f(h+k) = f(h) + f(k) + O( \eps \|c\|_\infty )$$
for all such $h,k$.  In particular, if $\eps$ is small enough, the function 
$$ c'(h,k) \coloneqq c(h,k) + f(h+k) - f(h) - f(k)$$
is such that $\|c'\|_\infty \leq \frac{1}{2} \|c\|_\infty$; also $\|f\|_\infty \leq \|c\|_\infty$.  Because $c$ obeys \eqref{cocycl}, the function $c'$ does also.  If one iterates this construction and sums the telescoping series, one obtains the claim.
\end{proof}

\subsection{Main argument}

We now establish the following proposition, which implies Lemma \ref{stab-poly} as a special case:

\begin{proposition}[Stability of polynomials]\label{propstab}  Let $G$ be a nonstandard finite additive group, let $H/\Gamma$ be a prefiltered nilmanifold of some degree $s \geq 0$, and let $\tilde g \colon G \to H/\Gamma$ be an almost polynomial map. Then there exists an internal polynomial map $g \colon G \to {}^*(H/\Gamma)$
such that
\begin{equation}\label{toast}
\tilde g(x) = \st g(x)
\end{equation}
for $\mu_G$-almost all $x$.
\end{proposition}

The rest of this section is devoted to a proof of this proposition. We induct on $s$. When $s=0$ then $\tilde g$ is almost always constant in the sense that $\aderiv_h g(x) = 1$ for $\mu_{G^2}$-almost all $h,x$ and the claim is clear by repeating the final part of the proof of Lemma \ref{near-const}, so suppose $s \geq 1$ and the claim has already been proven for $s-1$.  

It is convenient to make some reductions regarding the nature of the prefiltered nilmanifold $H/\Gamma$:

\begin{proposition}  To establish Proposition \ref{propstab} for a given choice of $s$, we may assume without loss of generality that $H=H_0=H_1$ (i.e., that the prefiltration on $H$ is a filtration) and that $\Gamma_s = \Gamma \cap H_s$ is trivial (and hence $H_s$ is now a central compact subgroup of $H$).
\end{proposition}

\begin{proof}
By construction, $\tilde g$ almost always takes values in $H_0/\Gamma$, so after adjusting $\tilde g$ on a $\mu_G$-null set and redefining $H$ we may assume without loss of generality that $H=H_0$.  Quotienting out the group $H_1$ (which is now normal), we conclude that $\tilde g$ is almost always constant modulo $H_1$, so by modification on a further null set we may assume that $\tilde g$ takes values on a single coset of $H_1$.  Applying a translation we can move this coset to the origin, and so after redefining $H=H_0$ we may assume that $H=H_0=H_1$.  Among other things, this implies that $H_s$ is now central.  One can now quotient both $H$ and $\Gamma$ by the central group $\Gamma \cap H_s$ without affecting the polynomial structure of $H/\Gamma$, so one may assume without loss of generality that $\Gamma_s = \Gamma \cap H_s$ is trivial, as claimed. 
\end{proof}

Henceforth we assume that the conclusions of the above proposition hold.  We now use the induction hypothesis to reduce the task of proving Proposition \ref{propstab} to that of solving a certain lifting problem. Let $\pi \colon H/\Gamma \to H/H_s\Gamma$ be the projection map, which is a nilspace morphism.  We can identify $H/H_s\Gamma$ with a prefiltered nilmanifold of degree $s-1$; one can think of $H/\Gamma$ as a principal $H_s$-bundle over $H/H_s\Gamma$.  The function $\pi \circ \tilde g$ is an almost polynomial map from $G$ into this nilmanifold, thus by induction hypothesis there is an internal polynomial map $\overline{g}: G \to {}^*(H/H_s \Gamma)$ such that 
\begin{equation}\label{pig}
\pi \circ \tilde g(x) = \st \overline{g}(x)
\end{equation}
for $\mu_G$-almost all $x \in G$.  Thus, if for every $x \in G$ we define the nonstandard $H_s$-fibre $F_x \subset {}^* (H/\Gamma)$ as
$$ F_x \coloneqq \{ y \in {}^* (H/\Gamma): {}^* \pi(y) = \overline{g}(x) \}$$
then $F_x$ depends internally on $x$, and from \eqref{pig} we see that $\tilde g(x) \in \st(F_x)$ for $\mu_G$-almost all $x \in G$.  Observe that each $F_x$ is a free orbit (or \emph{torsor}) of ${}^* H_s$: thus if $y, z \in F_x$ then there is a unique $h_s \in {}^* H_s$ such that $h_s y = z$.  Let ${\mathcal S}$ denote the (internal) set of all internal functions $g \colon G \to {}^* H/\Gamma$ that are ``lifts'' of $\overline{g}$ in the sense that $g(x) \in F_x$ for all $x$.  This is a torsor of the internal functions from $G$ to ${}^* H_s$; indeed, if $g \in {\mathcal S}$ and $h_s \colon G \to {}^* H_s$ is internal then $h_s g \in {\mathcal S}$, and conversely if $g,g' \in {\mathcal S}$ then there is a unique internal $h_s \colon G \to {}^* H_s$ such that $g' = h_s g$.  It will suffice to find an element $g \in {\mathcal S}$ that is an internal polynomial and such that $\st g(x) = \tilde g(x)$ for $\mu_G$-almost all $x \in G$.

To find this lift we proceed in several stages.  We first establish a weaker claim in which the requirement of being an (internal) polynomial is removed:

\begin{lemma}[Existence of a non-polynomial lift near $\tilde g$]\label{nlift} There exists $g^{**} \in {\mathcal S}$ (not necessarily a polynomial) such that $\st g^{**}(x) = \tilde g(x)$ for $\mu_G$-almost all $x \in G$.  
\end{lemma}

\begin{proof}  As $\tilde g$ is Loeb measurable, we can find an internal function $g^* \colon G \to {}^* H/\Gamma$ such that $\tilde g(x) = \st g^*(x)$ for $\mu_G$-almost all $x$, and hence 
$$ \st \overline{g}(x) = \st {}^* \pi(g^*(x))$$
for $\mu_G$-almost all $x \in G$.  Since $\overline{g}$ and ${}^* \pi \circ g^*$ are both internal functions, we can thus write
$$ \overline{g}(x) = \overline{c}(x) {}^* \pi(g^*(x))$$
for all $x \in G$, where $\overline{c} \colon G \to {}^* (H/H_s)$ is an internal function with $\st \overline{c}(x)=1$ for $\mu_G$-almost all $x \in G$.  We can lift this up to an internal function $c \colon G \to {}^* H$ with $\st c(x)=1$ for $\mu_G$-almost all $x \in G$; if we then define $g^{**} \coloneqq c g^*$, we see that $g^{**}$ obeys the stated properties.
\end{proof}

Next, we identify the set ${\mathcal S}$ with a certain type of group homomorphism.  For any $x,h \in G$, let $\Hom( F_x \to F_{x-h} )$ denote the space of internal functions $\phi \colon F_x \to F_{x-h}$ which are ${}^* H_s$-equivariant in the sense that $\phi(h_s y) = h_s \phi(y)$ for any $h_s \in {}^* H_s$ and $y \in F_x$.  These spaces give the fibration $(F_x)_{x \in G}$ the structure of a groupoid, since we obviously have a composition law
$$ \circ \colon\Hom(F_{x-h} \to F_{x-h-k}) \times \Hom(F_x \to F_{x-h})\to \Hom(F_x \to F_{x-h-k})$$
for any $x,h,k \in {}^* G$, obeying the following groupoid axioms:
\begin{itemize}
    \item (Identity)  The identity map $\mathrm{id}_x \in \Hom(F_x \to F_x)$ is such that $\phi \circ \id_x = \id_{x-h} \circ \phi = \phi$ for all $x,h \in G$ and $\phi \in \Hom(F_x \to F_{x-h})$.
    \item (Associativity)  For any $x,h,k,l \in G$ and $\phi_1 \in \Hom(F_{x-h-k} \to F_{x-h-k-l})$, $\phi_2 \in \Hom(F_{x-h} \to F_{x-h-k})$, $\phi_3 \in \Hom(F_x \to F_{x-h})$, one has $\phi_1 \circ (\phi_2 \circ \phi_3)) = (\phi_1 \circ \phi_2) \circ \phi_3$.
    \item (Inverse)  For any $x, h \in G$ and $\phi \in \Hom(F_x \to F_{x-h})$, there exists $\phi^{-1} \in \Hom(F_{x-h} \to F_x)$ such that $\phi \circ \phi^{-1} = \id_{x-h}$ and $\phi^{-1} \circ \phi = \id_x$. 
\end{itemize}

.  For each $x \in G$, one can canonically identify $\Hom(F_x \to F_x)$ with the group ${}^* H_s$, since the ${}^* H_s$-equivariant automorphisms of $F_x$ are given precisely by left multiplication by elements of ${}^* H_s$.  Note that if $y \in F_x$ and $z \in F_{x-h}$ there is a unique element of $\Hom(F_x \to F_{x-h})$ that maps $y$ to $z$; we will denote this map by ${}^* \pi^\Box(y,z)$ (the reason for the notation $\pi^\Box$ will be clearer later).  Thus we have
\begin{equation}\label{piinv}
{}^*  \pi^\Box(y,z) = {}^* \pi^\Box(h_s y, h_s z)
\end{equation}
for any $h_s \in {}^* H_s$, $y \in F_x$, and $z \in F_{x-h}$, and
\begin{equation}\label{pizw}
{}^*  \pi^\Box(z,w) \circ {}^* \pi^\Box(y,z) = {}^* \pi^\Box(y,w)
\end{equation}
whenever $y \in F_x$, $z \in F_{x-h}$, and $w \in F_{x-h-k}$.

Now let ${\mathcal G}$ denote the collection of all pairs $(h, \phi)$ where $h \in G$ and $\phi \colon G \mapsto \biguplus_{x \in G} \Hom(F_x \to F_{x-h})$ is an internal function such that  $\phi(x) \in \Hom(F_x \to F_{x-h})$ for all $x \in G$.  We can define a group law
\begin{equation}\label{gan}
 (k, \psi) (h, \phi) \coloneqq (h+k, \psi(\cdot-h) \circ \phi).
 \end{equation}
One easily verifies that this gives ${\mathcal G}$ the structure of an (internal) group.  Observe that for any $g \in {\mathcal S}$, one can define an internal group homomorphism $\aderiv g \colon G \to {\mathcal G}$ by the formula
$$ \aderiv g(h) \coloneqq (h, \aderiv_h g)$$
where the function $\aderiv_h g$ is defined by $\aderiv_h g(x) \coloneqq {}^* \pi^\Box( g(x), g(x-h) )$ for any $x,h \in G$.  From \eqref{pizw} one easily verifies that $\aderiv g$ is indeed an internal homomorphism; from \eqref{piinv} we also see that $\aderiv(h_s g) = \aderiv(g)$ whenever $h_s \in {}^* H_s$ is a constant.  In the converse direction, if one has an internal group homomorphism $h \mapsto (h,\phi_h)$ from $G$ to ${\mathcal G}$, then we claim that this homomorphism must be of the form $\aderiv g$ for some $g \in {\mathcal S}$, which is unique up to left-multiplication by constants $h_s \in {}^* H_s$.  Indeed, if we let $y_0$ be an arbitrary element of $F_0$, we can define $g(h)$ to be the unique element of $F_h$ such that
\begin{equation}\label{phih}
 \phi_{-h}(0) = \pi^\Box( y_0, g(h) )
 \end{equation}
(i.e., $g(h)$ is the image of $y_0$ under the equivariant map $\phi_{-h}(0) \in \Hom(F_0 \to F_h)$).  For any $h,k \in G$, one sees from \eqref{gan} and the group homomorphism hypothesis that
$$ \phi_{h+k}(0) = \phi_k(-h) \circ \phi_h(0)$$
and thus by \eqref{pizw}, \eqref{phih}
$$ \phi_k(-h) = {}^* \pi^\Box(g(-h), g(-h-k)).$$
From this and a change of variables we see that $\aderiv g(h) = (h, \phi_h)$, giving existence.  If we have $\aderiv g = \aderiv g'$ for some other $g' \in {\mathcal S}$, then writing $g' = h_s g$ for some internal $h_s \colon G \to {}^* H_s$, we see from (the converse of) \eqref{piinv} that $h_s(x-h) = h_s(x)$ for all $x,h \in G$, and hence $h_s$ is constant, giving the uniqueness claim.

Let $g^{**}$ be as in Lemma \ref{nlift}.
Define the external subgroup ${\mathcal G}^{**}$ of ${\mathcal G}$ to be the set of all pairs $(h,\phi) \in {\mathcal G}$ such that $\st \phi(x) = \st \aderiv_h g^{**}(x)$ for $\mu_G$-almost all $x \in G$.  Thus for instance $(h, \aderiv_h g^{**}) \in {\mathcal G}^{**}$ for every $h \in G$. It is easy to verify that ${\mathcal G}^{**}$ is an external subgroup of ${\mathcal G}$.

It will now suffice to show

\begin{proposition}[Existence of an exact group homomorphism]\label{exact-prop}  There exists an internal group homomorphism from $G$ to ${\mathcal G}^{**}$ of the form $h \mapsto (h,\phi_h)$, which is also ``exact'' in the sense that it is of the form $\aderiv g$ for some internal polynomial $g \in {\mathcal S}$. 
\end{proposition}

Indeed, $\phi_h, g$ obeyed the conclusions of Proposition \ref{exact-prop}, then by Lemma \ref{fubini}(ii) there exists $x_0 \in G$ such that $\st \aderiv_h g(x_0) = \st \phi_h(x_0) = \st \aderiv_h g^{**}(x_0)$ for $\mu_G$-almost all $h \in G$.  Writing $g = h_s g^{**}$ for some internal $h_S \colon G \to {}^* H_s$, this implies that $\st h_s(x_0-h) = \st h_s(x_0)$ for $\mu_G$-almost all $h \in G$, thus the internal polynomial map $g' \coloneqq h_s(x_0)^{-1} g$ is an element of ${\mathcal S}$ with $\st g'(x) = \st g^{**}(x) = \tilde g(x)$ for $\mu_G$-almost all $x \in G$.

It remains to establish Proposition \ref{exact-prop}.  The next step is to interpret the internal maps $\phi \colon G \mapsto \biguplus_{x \in G} \Hom(F_x \to F_{x-h})$ appearing in the definition of ${\mathcal G}$ as maps between filtered spaces, using the algebraic formalism from \cite[\S 7]{gt1}, in order to reformulate the exactness condition in a more tractable form.  We form the prefiltered group $H^\Box$ to be the group $H \times H$ equipped with the prefiltration
$$ (H^\Box)_i \coloneqq \{ (h_i, h_{i+1} h_i): h_i \in H_i, h_{i+1} \in H_{i+1}\}.$$
The significance of this group is that we have the identity
\begin{equation}\label{hhh}
\HK^{k+1}(H) = \HK^k(H^\Box)
\end{equation}
for all $k \geq 0$, by using the identification
$$ (h_\omega)_{\omega \in \{0,1\}^{k+1}} \equiv ((h_{\omega,0}, h_{\omega,1}))_{\omega \in \{0,1\}^k}.$$
Indeed, this can be checked by noting that every generator of $\HK^{k+1}(H)$ can be expressed (using this identification) as a finite combination of generators of $\HK^k(H^\Box)$, and conversely.  We define $\Gamma^\Box$ similarly.  It is easy to verify that $H^\Box/\Gamma^\Box$ is a prefiltered nilmanifold of degree $s$ (as a nilmanifold it is just the Cartesian square of $H/\Gamma$).  Note that $H^\Box_s = \{ (h_s,h_s): h_s \in H_s \}$ is a central subgroup of $H^\Box$ (and $\Gamma^\Box_s$ is trivial).  We let $\pi^\Box \colon H^\Box/\Gamma^\Box \to H^\Box/(H^\Box)_s \Gamma^\Box$ be the projection map; this is a nilspace morphism to a prefiltered nilmanifold $H^\Box/(H^\Box)_s \Gamma^\Box$ of degree $s-1$.

Observe that if $y \in F_x$ and $z \in F_{x-h}$ for some $x,h \in G$, then ${}^* \pi^\Box(y,z)$, by definition, is the equivalence class $\{ (h_s y, h_s z) \colon h_s \in {}^* H_s \}$.  By abuse of notation we may therefore identify ${}^* \pi^\Box(y,z)$ with the unique element of $\Hom(F_x \to F_{x-h})$ mapping $y$ to $z$, thus making this use of the notation ${}^* \pi^\Box(y,z)$ consistent with the one used earlier.  In particular, we can view $\biguplus_{x \in G} \Hom(F_x \to F_{x-h})$ as an internal subset of the internal filtered nilmanifold ${}^* \pi^\Box(H^\Box/\Gamma^\Box)$, and it now makes sense to ask when the map $\phi \colon G \mapsto \biguplus_{x \in G} \Hom(F_x \to F_{x-h})$ is an internal polynomial or not.

We define ${\mathcal G}_0$ to be the internal subgroup of ${\mathcal G}$ consisting of those pairs $(h,\phi) \in {\mathcal G}$ with $\phi$ an internal polynomial.  It is easy to see that this is indeed an internal subgroup.

We now make the following key observations, which generalise the familiar fact that a (differentiable) function of one variable is a polynomial of degree $s$ if and only if its derivative is a polynomial of degree $s-1$:

\begin{itemize}
\item[(i)] If $g \in {\mathcal S}$ is a polynomial, then $\aderiv g$ takes values in ${\mathcal G}_0$ (i.e., $\aderiv_h g$ is a polynomial for every $h \in G$).
\item[(ii)] Conversely, if $g \in {\mathcal S}$ is such that $\aderiv g$ takes values in ${\mathcal G}_0$,
then $g$ is a polynomial.
\end{itemize}

The direction (i) is immediate from \eqref{hhh} (and its counterpart for $\Gamma$).  For the converse direction (ii), suppose that $\aderiv_h g$ is a polynomial for every $h \in G$.  Expanding out the definitions, we conclude in particular that
$$
\left( {}^* \pi^\Box( (g(x+\omega \cdot h), g(x+\omega \cdot h + h_{s+1})) ) \right)_{\omega \in \{0,1\}^s} \in \HK^s( {}^* \pi^\Box( H^\Box / \Gamma^\Box ) )$$
for all $x,h_1,\dots,h_{s+1} \in G$, where $h \coloneqq (h_1,\dots,h_s)$. Since $\pi^\Box$ quotients out $H^\Box/\Gamma^\Box$ by the central order $s$ group $(H^\Box)_s$, we conclude from the definition of $\HK^s$ that
$$
\left(  (g(x+\omega \cdot h), g(x+\omega \cdot h + h_{s+1})) \right)_{\omega \in \{0,1\}^s} \in \HK^s( H^\Box / \Gamma^\Box  ).$$
By \eqref{hhh} this implies that
$$
\left(  g(x+\omega' \cdot h' ) \right)_{\omega' \in \{0,1\}^{s+1}} \in \HK^{s+1}( H / \Gamma )$$
for all $x,h_1,\dots,h_{s+1} \in G$, where $h' \coloneqq (h_1,\dots,h_{s+1})$.  Because $H/\Gamma$ is a nilspace of degree at most $s$ (see \cite[Proposition 3.10]{gmv}), this implies that
$$
\left(  g(x+\omega_1 h_1 + \dots + \omega_k h_k ) \right)_{\omega_1,\dots,\omega_k \in \{0,1\}} \in \HK^k( H / \Gamma )$$
for any standard $k \geq 0$ and $x,h_1,\dots,h_{k} \in G$.  Thus $g$ is a polynomial as required.

In view of the above equivalence, the task of proving Proposition \ref{exact-prop} (and hence Proposition \ref{propstab}) now reduces to establishing

\begin{proposition}[Existence of a group homomorphism]\label{hom-prop}  There exists an internal group homomorphism from $G$ to ${\mathcal G}_0 \cap {\mathcal G}^{**}$ of the form $h \mapsto (h,\phi_h)$. 
\end{proposition}

By hypothesis, $\tilde g$ is an almost polynomial, hence $\st g^{**}$ is also (since these two maps agree $\mu_G$-almost everywhere).  From \eqref{hhh} and Lemma \ref{fubini}(ii) we conclude that $\st \aderiv_h g^{**} \colon G \to \pi^\Box(H^\Box/\Gamma^\Box)$ is an almost polynomial map for $\mu_G$-almost all $h \in G$.  Applying the induction hypothesis (and the axiom of choice), we conclude 

\begin{proposition}[First approximation to $\phi_h$]\label{first-phi}  For $\mu_G$-almost all $H \in G$, there exists an internal polynomial $\phi'_h \colon G \to {}^* \pi^\Box(H^\Box/\Gamma^\Box)$ such that $\st \aderiv_h g^{**}(x) = \st \phi'_h(x)$ for $\mu_G$-almost all $x \in G$.  In particular, $(h, \phi'_h) \in {\mathcal G}_0$.  
\end{proposition}

This begins to look like what we want; however, there are currently three undesirable features of this map $\phi'_h$ that prevent us from simply taking $\phi_h$ to equal $\phi'_h$:

\begin{itemize}
\item[(i)] It is not necessarily the case that $(h,\phi'_h)$ lies in ${\mathcal G}_0$, because $\phi'_h(x)$ need not lie in $\Hom(F_x \to F_{x-h})$.
\item[(ii)]  Even if (i) was resolved, it is not necessarily the case that $\phi'_h$ depends in an internal fashion on $h$ (due to the use of the axiom of choice).
\item[(iii)] Even if (i) and (ii) were resolved, it is not necessarily the case that the map $h \mapsto (h,\phi'_h)$ is a group homomorphism.
\end{itemize}

We now address each of these issues in turn.  We begin by resolving (i):

\begin{proposition}[Second approximation to $\phi_h$]\label{second-phi} For almost all $h \in G$, there exists an internal polynomial $\phi''_h \in \Hom(F_x \to F_{x-h})$ such that $\st \aderiv_h g^{**}(x) = \st \phi''_h(x)$ for $\mu_G$-almost all $x \in G$.  In particular, $(h, \phi''_h) \in {\mathcal G}_0 \cap {\mathcal G}^{**}$.
\end{proposition}

\begin{proof}  Let $\phi'_h$ be as in Proposition \ref{first-phi}.
By Lemma \ref{fubini}(ii), we know that for $\mu_{G^2}$-almost all $(x,h) \in G$, that $\st \aderiv_h g^{**}(x) = \st \phi'_h(x)$.  Since $\aderiv_h g^{**}(x) \in \Hom(F_x \to F_{x-h})$, we thus see that
$$ \st {}^* \tilde \pi( \phi'_h(x) ) = \st {}^* \tilde \pi( \aderiv_h g^{**}(x) ) = \st (\overline{g}(x), \overline{g}(x-h))$$
where $\tilde \pi \colon \pi^\Box(H^\Box/\Gamma^\Box) \to \pi(H/\Gamma) \times \pi(H/\Gamma)$ is the projection map
$$ \tilde \pi \colon \pi^\Box(y,z) \mapsto (\pi(y), \pi(z)).$$
We conclude that 
$$ {}^* \tilde \pi( \phi'_h(x) ) =  (\overline{c}_1(x,h) \overline{g}(x), \overline{c}_2(x,h) \overline{g}(x-h))$$
for all $x,h \in G$ and some internal functions $\overline{c}_1, \overline{c}_2 \colon G \times G \to {}^* \pi(H)$ with the property that $\st \overline{c}_1(x,h) = \st \overline{c}_2(x,h)=1$ for $\mu_{G^2}$-almost all $(x,h) \in G$. Using some arbitrary continuous local lift of $\pi(H)$ up to $H$ around the identity, we can write $\overline{c}_i = \pi(c_i)$ for $i=1,2$ and some internal functions $c_1,c_2 \colon G \times G \to {}^* H$ with $\st c_1(x,h) = \st c_2(x,h) = 1$ for $\mu_{G^2}$-almost all $(x,h) \in G$.  If we then set
$$ \phi''_h(x) \coloneqq {}^* \pi^\Box(c_1(x,h), c_2(x,h))^{-1} \phi'_h(x)$$
then $(h,\phi''_h)$ is now an element of ${\mathcal G}_0$ for every $h$, and $\st \aderiv_h g^{**}(x) = \st \phi'_h(x) = \st \phi''_h(x)$ for $\mu_{G^2}$-almost all $(x,h) \in G^2$, so that $(h, \phi''_h) \in {\mathcal G}^{**}$ for $\mu_G$-almost all $h \in G$.  The claim follows.
\end{proof}

Now we address the difficulty (ii).  

\begin{proposition}[Third approximation to $\phi_h$]\label{third-phi} There exists an internal subset $E$ of $G$ of full Loeb measure and an internal map $h \mapsto (h,\phi'''_h)$ from $E$ to ${\mathcal G}_0 \cap {\mathcal G}^{**}$.
\end{proposition}

\begin{proof}
Let $\phi''_h$ be as in Proposition \ref{second-phi}.
By Lemma \ref{fubini}(ii), there exists $x_0 \in G$ such that $\st \aderiv_h g^{**}(x_0) = \st \phi''_h(x_0)$ for $\mu_G$-almost all $h \in G$.  Since $\aderiv_h g^{**}(x_0)$ and $\phi''_h(x_0)$ both lie in $\Hom(F_{x_0} \to F_{x_0-h})$, there exists an internal function $c \colon G \to {}^* H_s$ such that $\aderiv_h g^{**}(x_0) = c(h) \phi''_h(x_0)$ for all $h \in G$; by the construction of $x_0$, we have $\st c(h)=1$ for $\mu_G$-almost all $h \in G$.  

Fix a metric $d$ on $\pi^\Box(H^\Box / \Gamma^\Box)$, and for any $\eps>0$, let $E_\eps \subset G$ be the set of all $h \in G$ such that there exists $(h, \phi'''_h) \in {\mathcal G}_0$ with $\phi'''_h(x_0) = \aderiv_h g^{**}(x_0)$, and such that
\begin{equation}\label{stad}
 {}^* d( \phi'''_h(x), \aderiv_h g^{**}(x) ) \leq \eps
\end{equation}
for a set of $x$ in $G$ of Loeb measure at least $1-\eps$.  Clearly $E_\eps$ is an internal subset of $G$; from the above discussion (choosing $\phi'''_h = c(h) \phi''_h$) we see that $E_\eps$ has internal measure at least $1-\eps$ for every standard $\eps>0$, hence by the underspill principle (see e.g., \cite[Corollary 5.2(iv)]{tz-concat}) we can find an infinitesimal $\eps>0$ such that $E = E_\eps$ has internal measure at least $1-\eps$, and hence full Loeb measure.  Fixing such an $\eps$, we now see from Corollary \ref{close} that if $h \in E$, then the function $\phi'''_h$ defined above is unique, so $\phi'''_h$ depends internally on $h$ on this domain.  By construction $h \mapsto (h,\phi'''_h)$ is now an internal map from $E$ to ${\mathcal G}_0 \cap {\mathcal G}^{**}$, and the claim follows. 
\end{proof}

Now we address the issue (iii).  

\begin{proposition}[Fourth approximation to $\phi_h$]\label{fourth-phi} There exists an internal subset $E$ of $G$ of full Loeb measure and internal homomorphism $h \mapsto (h,\phi''''_h)$ from $E$ to ${\mathcal G}_0 \cap {\mathcal G}^{**}$.
\end{proposition}

\begin{proof} Let $\phi'''_h$ and $E$ be as in Proposition \ref{third-phi}.
If $h,k,h+k \in E$, we have
\begin{equation}\label{coco}
(h+k, \phi'''_{h+k}) = (0, c(h,k)) (h, \phi'''_h) (k, \phi'''_k)
\end{equation}
for some $(0,c(h,k)) \in {\mathcal G}_0 \cap {\mathcal G}^{**}$, and from \eqref{stad} we see that $\st c(h,k)(x) = 1$ for $\mu_G$-almost all $x \in G$.  Applying Corollary \ref{close}, we conclude that $c(h,k) \in {}^* H_s$ is a constant independent of $x$ with $\st c(h,k) = 1$.  In particular the group element $(0,c(h,k))$ appearing in \eqref{coco} is central in ${\mathcal G}_0$.

Now suppose that $x,h,k,l$ are such that $h,k,l,h+k,k+l,h+k+l \in E$.  By several applications of \eqref{coco} we conclude that $(h+k+l,\phi'''_{h+k+l})$ is equal to both
$$ (0, c(h,k) c(h+k,l)) (l,\phi'''_l) (k,\phi'''_k) (h,\phi'''_h)$$
and to
$$ (0, c(h,k+l) c(k,l)) (l,\phi'''_l) (k,\phi'''_k) (h,\phi'''_h)$$
leading to the cocycle equation
$$ c(h,k) c(h+k,l) = c(h,k+l) c(k,l)$$
whenever $h,k,l,h+k,k+l,h+k+l \in E$.  Taking logarithms (using the fact that $\st c(h,k)= 1$ and that $H_s$ is a abelian Lie group) and using Lemma \ref{cocy-triv}, we conclude that there exists an internal function $f \colon E \to {}^* H_s$ such that
$$ c(h,k) = f(h+k) f(h)^{-1} f(k)^{-1}$$
whenever $h,k,h+k \in E$, and also $\st f(h) = 1$ for all $h \in E$.  If one then defines
$$ \phi''''_h(x) \coloneqq f(h)^{-1} \phi'''_h(x)$$
for $h \in E$, we now see that $h \mapsto (h,\phi''''_h)$ is an internal map from $E_\eps$ to ${\mathcal G}_0 \cap {\mathcal G}^{**}$ with the property that
$$ (h+k, \phi''''_{h+k}) = (k, \phi''''_k) (h, \phi''''_h)$$
whenever $h,k,h+k \in E$.
\end{proof}

By Proposition \ref{fourth-phi} and Lemma \ref{repair} we thus obtain the required internal homomorphism $h \mapsto (h, \phi_h)$ from $G$ to ${\mathcal G}_0 \cap {\mathcal G}^{**}$, giving Proposition \ref{exact-prop}.  The proof of Proposition \ref{propstab} is now complete.

\section*{Acknowledgments} 

We thank Uwe Stroisnki and James Leng for corrections, and the anonymous referee for useful suggestions.

\bibliographystyle{amsplain}


\begin{dajauthors}
\begin{authorinfo}[aj]
  Asgar Jamneshan\\
	Ko\c{c} University\\
College of Sciences\\
Rumeli Feneri Yolu\\
34450 Sar\i yer\\
$\dot{\mathrm I}$stanbul, T\"urkiye\\
	ajamneshan\imageat{}ku\imagedot{}edu\imagedot{}tr\\
  \url{https://asgarjam.wixsite.com/home}
\end{authorinfo}
\begin{authorinfo}[tao]
  Terence Tao\\
  Department of Mathematics, UCLA\\
  Los Angeles, CA 90095\\
  USA\\
  tao\imageat{}math\imagedot{}ucla\imagedot{}edu \\
  \url{https://www.math.ucla.edu/\~tao}
\end{authorinfo}
\end{dajauthors}

\end{document}